\documentclass[a4paper,12pt,oneside,openbib]{article}
\usepackage{mathrsfs}%used it for Fourier Function
\usepackage{amsmath}
\usepackage{moreverb}
\usepackage{vmargin}
\usepackage{pdfpages}
\usepackage{bm}
\usepackage{bbm}
\usepackage{latexsym,amsfonts,amssymb,amscd}
\usepackage[T1]{fontenc}
\usepackage{amsthm}
\usepackage{tcolorbox}
\usepackage{dsfont} %for indicator 1
\usepackage{etoolbox}
\usepackage[square, numbers, comma, sort&compress]{natbib} 
\usepackage{verbatim}
\usepackage{bm}
\usepackage[T1]{fontenc}
\usepackage{tikz}
\usepackage[english]{babel}
\usepackage{pdfpages}
\usepackage{fancyhdr, fancybox} 
\usepackage{moreverb}
\usepackage{vmargin}
\usepackage{tcolorbox}
\usepackage{abstract}
\numberwithin{equation}{section} %numbering equations now depend on the section
\newcommand{\RomanNumeralCaps}[1]
 {\MakeUppercase{\romannumeral #1}}
 \newcommand{\R}{{ I\!\!R}}
 \newcommand{\N}{{I\!\!N}}
 \newcommand{\Z}{\mathbb Z}
 \renewcommand{\H}{{\mathcal H}}
\newtheorem{theorem}{Theorem}[section]

\newtheorem{lemma}[theorem]{Lemma}
\newtheorem{proposition}[theorem]{Proposition}

\newtheorem{remark}[theorem]{Remark}

\newcommand{\tendsto}[1]{\renewcommand{\arraystretch}{0.5}
\begin{array}[t]{c}
\longrightarrow \\
{ \scriptstyle #1 }
\end{array}
\renewcommand{\arraystretch}{1}}
\newcommand{\weaktendsto}[1]{\renewcommand{\arraystretch}{0.5}
\begin{array}[t]{c}
\rightharpoonup \\
{ \scriptstyle #1 }
\end{array}
\renewcommand{\arraystretch}{1}}
\usepackage{authblk}

\date{}

\author[1]{May ABDALLAH 
\thanks{May\textunderscore{}i\textunderscore{}abdallah@outlook.com}}

\author[2]{Mohamad DARWICH \thanks{Mohamad.Darwich@lmpt.univ-tours.fr}}

\author[3] {Luc MOLINET \thanks{Luc.Molinet@univ-tours.fr}}

\affil[1,3] {Institut Denis Poisson, Universit\'e de Tours, Universit\'e d'Orl\'eans, CNRS, Parc Grandmont, 37200 Tours, France}
\affil[1,2] {Laboratory of Mathematics-EDST, Department of Mathematics, Faculty of Sciences \RomanNumeralCaps{1}, Lebanese University, Beirut, Lebanon}

\begin{document}
\date{ }
\title{\bf Asymptotic stability of fast solitary waves to  the  Benjamin Equation}
\maketitle

%\newpage\thispagestyle{empty}\addtocounter{page}{-1}
%\section*{Acknowledgment}

%\newpage\thispagestyle{empty}\addtocounter{page}{-1}
%~
%\newpage\thispagestyle{empty}\addtocounter{page}{-1}
%\cleardoublepage
%\phantomsection

\begin{abstract}
We prove the asymptotic stability of the high speed solitary waves to the Benjamin equation.  This is done by establishing a Liouville property  for the nonlinear evolution of the Benjamin equation around  these  solitary waves. To do this, inspired by \cite{Kenig Martel Robbiano 2011}, we make use of the KdV limit of the Benjamin equation together with known rigidity property of the KdV flow. The main difficulties are linked to the presence of the non-local operator $ \H $ as well as the non-positivity of the quadratic part of the energy in the case $ \gamma<0$ which is the physical case.
\end{abstract}

%\newpage
\section{Introduction}\label{Intro3}
The Benjamin equation reads
\begin{equation}\label{MainEq}
\partial_t u + \ \partial_x^3 u  +  \gamma \, \H  \partial_x^2 u + u \partial_x u = 0
\end{equation}
where $ \gamma\in \R^* $, and $\H$ denotes the Hilbert transform defined as  the Fourier multiplier by $ i \xi \,\text{sgn}(\xi) $.  Note that with this definition $ \H u_x=-D_x u $ where $ D_x $ is the Fourier multiplier by $ |\xi| $.

This equation, with $ \gamma<0 $,  has been derived in \cite{Benjamin 1992} 
(see also  \cite{Benjamin 1996} and \cite{ABR}) as an asymptotic model for  the propagation    of long interfacial waves when the surface tension cannot be safely ignored. In this context,  $ u $ represents the vertical displacement of the interface between the two fluids. 

It is worth noticing that it reduces to the Korteweg-de Vries (KdV) equation
$$ u_t +  u_{xxx} +  u u_x =0$$
when $ \gamma=0 $. However,  whereas  the KdV  equations is a completely integrable systems,  the Benjamin equation is not known to be integrable. 

The existence, non existence, uniqueness  and orbital stability of solitary waves to \eqref{MainEq} has been studied in numerous papers (\cite{Benjamin 1996}, \cite{Benett}, \cite{Angulo},  \cite{Klein et al 2023}, \cite{ADM1}, \cite{HSS24}). To summarize,  solitary waves are known to exist for $ c>0 $ when $ \gamma>0 $ and for $ c>\gamma^2/4 $ when $ \gamma<0 $ whereas non existence is proven for $ c\le 0 $ when $ \gamma>0 $ and for $ c\le -\gamma^2/5 $ when $ \gamma<0 $. In particular, the existence or non existence of solitary waves for $ c\in ]-\gamma^2/5,\gamma^2/4] $ when $ \gamma<0 $ still seems to be an open question. The orbital stability and a kind  of uniqueness result of the solitary wave is proven for  $ c>0$ large  with respect to $\gamma^2 $  without any sign condition on $ \gamma $ and for $ c>0 $ small enough with respect to $\gamma^2 $ when $ \gamma>0 $. In the rest of the range of existence, the orbital stability of the set of ground states is proved. 

In this paper, our aim is to prove that fast enough solitary waves are actually asymptotically stable in the energy space $ H^1(\mathbb{R}) $ without any sign condition on $ \gamma $. It is well known that the Cauchy problem associated with \eqref{MainEq} is globally well-posed in $H^1(\mathbb{R})$ (see \cite{Kenig Ponce Vega 1991}) and even in $ H^s(\mathbb{R}) $ for $s\ge 0$ (see \cite{linares}). 
 We recall that \eqref{MainEq} enjoys the following conservation laws
$$ \text{Energy:} \quad E(u(t)) = \int_{\mathbb{R}} \bigg(  \frac{1}{2} (\partial_x u)^2+ \frac{\gamma}{2} (D_x^{\frac{1}{2}} u )^2 - \frac{1}{6} u^3 \bigg) dx \quad  = E(u(0)),$$
and 
$$ \text{Mass:} \quad M(u(t)) = \frac{1}{2} \ \int_{\mathbb{R}} u^2 dx \quad = M(u(0)).$$
Note, in particular, that the quadratic part of the energy is not positively defined when $ \gamma<0 $. 
A solitary wave solution to \eqref{MainEq} is any traveling wave solution of the form
$$ u(t,x)= \varphi_{\gamma,c}(x - x_0 - c t)   \in C(\R;H^1(\R)) $$
where $c>0, \ x_0 \in \mathbb{R}$, and $\varphi_{\gamma,c} \in H^1(\R) $ solves
\begin{equation}\label{eqQ}
 c \ \varphi_{\gamma,c}- \ \varphi_{\gamma,c}^{''} - \gamma \ \H \varphi_{\gamma,c}^{'} - \frac{1}{2} \varphi_{\gamma,c}^2 =0 \ .
 \end{equation}
 In \cite{Benett} and later in \cite{ADM1}, it is proven that there exists $ r_2>0 $ such that for $ \gamma\in \R^* $ and $ c>r_2 \gamma^2 $, \eqref{eqQ} has a solution $ \varphi_{\gamma,c}\in H^\infty(\R)$ 
 that is orbitally stable in $ H^1(\R) $ and unique in some sense. More precisely, the uniqueness in \cite{Benett}, inducted  by the Implicit Function Theorem, is a uniqueness in an $H^1_e(\R) $\footnote{ $H^1_e(\R) $ stand for the closed subspace of $ H^1(\R) $ composed of even function} neighborhood of the soliton of speed $ c$ of the KdV equation whereas the uniqueness in \cite{ADM1} is a uniqueness among the even ground state solutions to \eqref{eqQ}, that is the solutions to \eqref{eqQ} that moreover
   solves the constraint minimization problem  
   $$ \inf \bigg\{ \frac{1}{2} \int_{\mathbb{R}} ( \varphi_x^2+\gamma (D_x^{\frac{1}{2}} \varphi)^2+c \varphi^2  ) dx\; ; \varphi \in H^1_e(\R) \; \text{such  that} \int_{\R} \varphi^3 = \lambda \bigg\} $$ for some $ \lambda>0$. However, it is quite easy to prove that at least for $ r_2>0 $ large enough these two solutions do coincide (see Proposition \ref{der} in Section \ref{subsectionsmooth}). In this work, we denote by $ Q_{\gamma,c}\in H^\infty(\R) $ this even smooth solution to \eqref{eqQ}.
    \begin{remark}
It is worth noticing that, in sharp contrast to the KdV or Benjamin-Ono soliton, numerical simulations in \cite{Klein et al 2023} seem to indicate that for $ \gamma<0 $, $ Q_{\gamma,c} $ has a negative part even for $ \gamma $ close to $0$.  This is due to the competition between the third and the second order operator in the linear part of the Benjamin equation.  

\end{remark}
 Our first result is a nonlinear Liouville property for  global solutions to \eqref{MainEq} with a uniform in time $ H^1$-decay  (up to space translation) in an $H^1$-neighborhood of $ Q_{\gamma,1} $  when $ |\gamma| $ small enough.

\begin{theorem}(Nonlinear Liouville Property)\label{NL}
There exists $ \varepsilon_0, \gamma_0>0 $ such that for any $ |\gamma|< \gamma_0$, any solution $u \in C(\mathbb{R}; H^1(\mathbb{R})) \cap L^{\infty}(\mathbb{R}, H^1(\mathbb{R}))$ of
\begin{equation}\label{eqbeta}
\partial_t u + u_{3x} +\gamma \H u_{xx} + u u_x = 0
\end{equation}
that satisfies 
\begin{enumerate}
\item there exists $ x_0 \, :\, \R \to \R $ such that 
\begin{equation}\label{hyp1}
\sup_{t\in\R} \|u(t)-Q_{\gamma,1}(\cdot-x_0(t))\|_{H^1} < \varepsilon_0 \ ,
\end{equation}
\item $ \forall \delta>0 $, $ \exists R>0 $ such that 
\begin{equation}\label{hyp2}
\forall t\in \R , \quad \int_{|x-x_0(t)|>R} u^2(t,x)+u_x^2(t,x) \, dx < \delta 
\end{equation}
\end{enumerate}
is of the form 
$$
u(t,\cdot)= Q_{\gamma,c}(\cdot-y-ct)) 
$$
 for some $ y\in \R $ and $ c>0 $ with $ |c-1|\lesssim \varepsilon $.
\end{theorem}
As a consequence of this nonlinear Liouville theorem, we obtain the following asymptotic stability result for fast enough even solitary waves to the Benjamin equation \eqref{MainEq}.
\begin{theorem}\label{Main}
There exists $ K>0 $ such that for any $ \gamma\in \R^* $, the  even  solitary wave $ Q_{\gamma,c}$ to \eqref{MainEq} (constructed in \cite{Benett} or \cite{ADM1}) with speed $ c>K \gamma^2 $ is $ H^1$-asymptotically stable. More precisely, for any such $ c $ there exists $ \varepsilon>0 $ such that if 
$$
\|u_0-Q_{\gamma,c}\|_{H^1} < \varepsilon 
$$
then there exist $c^* $ with  $ |c-c^*|\ll c $ and  a \ $ C^1$-function $ x :\, \R \to \R $ \ with $ \lim_{x\to+\infty} \dot{x}(t)=c^* $ such that the unique solution $ u\in C(\R;H^1(\R)) $ to \eqref{MainEq} satisfies 
$$
u(t,\cdot+x(t)) \weaktendsto{t\to +\infty} Q_{\gamma,c^*} \quad \text{in} \; H^1(\R)  \; .
$$
Moreover, %for any $ s<1 $,
$$
\lim_{t\to+\infty} \|u(t)-Q_{\gamma,c^*}(\cdot-x(t))\|_{H^1(]\frac{c\, t}{2},+\infty[)} =0 .
$$
In particular, $u(t)-Q_{\gamma,c^*}(\cdot-x(t)) $ converges uniformly to $0 $ on $ ]\frac{c \, t}{2},+\infty[$.
\end{theorem}
To prove our result, we start by establishing almost monotonicity estimates based on the conservation of the mass  and the energy.  These estimates are technically complicated here due to the non-local nature of the operator $ \H $ that can moreover be assigned with a positive or a negative sign. To overcome this difficulty, we follow the approach introduced in \cite{MMT} (see also \cite{Gustafson et al. 2009} where an almost  monotonicity result associated to the $ L^2$-norm for the Benjamin-Ono equation is proved) by  using  cut-off functions whose supports expand sub-linearly at the rate $O(t^{3/4})$. We will particularly need some commutator estimates  including one that improves a similar estimate in  \cite{Gustafson et al. 2009}. With these monotonicity results in hand, we are able to prove that any $ H^1$-compact global solution to \eqref{MainEq} has an $ H^1$-norm that decays at least as  $ \, |x|^{-1/4} $. At this stage, it is worth noticing that unfortunately our monotonicity result does not enable us to get the  optimal decay rate for  $ L^2$-compact global solutions. 

Then, we prove the nonlinear Liouville property for $ |\gamma|$ small enough  by a contradiction argument making use of the linear Liouville property of the KdV equation around the soliton (cf. \cite{Martel 2006}). This type of approach was  introduced in \cite{Kenig Martel Robbiano 2011}. Note that we have some  additional difficulties here. First, we cannot get a global estimate by the so-called Kato smoothing effect since we do not prove sufficient decay on compact global solutions. This is actually  the reason why we have to work in the $ H^1$-level. Our argument then combines the $ H^1$-almost monotonicity result with local Kato smoothing effect and an estimate on the derivative of the energy. Second, the equation has no dilation symmetry. Therefore, to normalize the solitary wave profile so that its $ L^2$-norm is the same as that of the initial data, we cannot just rely on  a dilation argument but we have to use a uniform lower bound for $\frac{d}{dc} \| Q_{\gamma,c}\|^2_{L^2} $. To obtain this lower bound, we rely on the implicit function theorem making use of the properties of the second derivative of the action functional $ E+cM $ associated to the KdV equation at the soliton profile. 

Finally, with the nonlinear Liouville theorem in hand, we prove the asymptotic stability result by a classical approach (see for instance \cite{MM1}, \cite{MM2}) with some additional technical difficulties due to the non-local operator $\H $ and the non-positivity of the quadratic part of the energy in the case $ \gamma<0 $.

\section{$L^2$-almost monotonicity to the right of a mainly localized solution }
In this section, we adapt the approach followed in \cite{Gustafson et al. 2009} to our situation.
We will frequently make  use of  the following commutator estimates:
\begin{lemma}\label{Commu Lemma}
There exists a constant $C >0$ such that if $f $ is any continuously differentiable function with $ f' \in L^\infty$ then
\begin{equation}\label{comut}
 \|[\H , f] \ u_x\|_{L^2} + \|[\H \partial_x, f] \ u\|_{L^2} \leq C \ \|f'\|_{L^\infty} \ \|u\|_{L^2} \ , \quad \quad \forall u\in L^2(\R) \; .
 \end{equation}
For  $0< \varepsilon \ll 1$, there exists a constant $ C>0 $ such that if $ f\in W^{3,\infty}(\R)$  then  
\begin{equation} \label{besov1}
 \|\partial_x([\H,f]u_x)\|_{L^2_x}  \lesssim \Bigl(  \|D_x^{2-\varepsilon}f\|_{L^\infty}+ \|D_x^{2+\varepsilon}f\|_{L^\infty}\Bigr) 
  \|u\|_{L^2} \ , \quad \forall u\in L^2(\R) \; .
\end{equation}
\end{lemma} 
The first estimate can be found in  \cite{Albert 1999} or [\cite{Albert Bona Saut 1961}, Lemma 3.2] whereas the proof of the second estimate is given  in the appendix.

In the sequel, we will make use of a non-negative, non-decreasing function $ \chi \in C^\infty_b(\R) $ such that  $x\mapsto \sqrt{\chi'(x)}$ belongs to  $C^\infty_b(\mathbb{R})$ with $\chi(x)=0$ if $x\le 0$, and $\chi(x)=1$ if $x\ge 1$. 

\begin{proposition} \label{L2Monot} 
Let $ |\gamma| \le 1/2  $ and $ u\in C(\R;H^1(\R)) $ be a solution to \eqref{MainEq} such that there exist $x\,:\, \R\to \R $ of class $ C^1 $ with 
   $ \inf_{\R} \dot{x}\ge 5/6 $  and $ R_0\ge 1  $ with
 \begin{equation}\label{loc}
\|u(t)\|_{L^\infty(|x-x(t)|>R_0)} \le2^{-6} \quad  \forall t\in\R .
 \end{equation}
  For  $ R>0 $, $ t_0\in\R $ and $\vartheta\in [\frac{3}{8},\frac{5}{8}]$,  we set 
     \begin{equation}\label{defI}
 I^{+ R}_{t_0} (t)=\frac{1}{2}\int_{\R} u^2(t) \chi \Bigl(\frac{\cdot-(x(t_0) + R +\vartheta (t-t_0))}{(R+\frac{1}{8}|t_0-t|)^{3/4}}\Bigr) \ .
 \end{equation} 
 and
  \begin{equation}\label{defI2}
 I^{- 2R}_{t_0} (t)=\frac{1}{2}\int_{\R} u^2(t) \chi \Bigl(\frac{\cdot-(x(t_0) - 2R +\vartheta (t-t_0))}{(R+\frac{1}{8}|t_0-t|)^{3/4}}\Bigr) \ .
 \end{equation} 
 Then we have 
 \begin{equation}
I^{+R}_{t_0}(t_0)-I^{+R}_{t_0}(t)\le K_0 R^{-1/4} , \quad \quad \forall t\le t_0 \quad  \label{mono}
\end{equation}
and 
 \begin{equation}
I^{-R}_{t_0}(t)-I^{-R}_{t_0}(t_0)\le K_0 R^{-1/4} , \quad \quad \forall t\ge t_0 \quad  \label{mono2}
\end{equation}
for some constant $ K_0>0 $ that only depends on $M(u)$ and $R_0$. 
\end{proposition}
\begin{proof}
 We start by  proving \eqref{mono}; the proof of \eqref{mono2} will follow exactly as the latter only up to slight modifications. We set 
 \begin{equation} \label{defPsi}
 \Psi(t,x) = \chi \Bigl(\frac{\cdot-(x(t_0) + R +\vartheta (t-t_0))}{(R+\frac{1}{8}(t_0-t))^\theta}\Bigr) 
 \end{equation}
 where $1/2<\theta<1 $ will be chosen later. 
 It is worth noticing that, for $ R\ge 1$,
 \begin{equation}\label{Psit22}
\Psi(t_0,x)=\chi\Bigl( \frac{x -x(t_0)-R)}{R^\theta}\Bigr)=1 \quad \text{ for } x\ge x(t_0)+2 R
 \end{equation}
 and for $ t\le t_0 $, 
 \begin{equation}\label{Psit}
\Psi(t,x)=\left\{ \begin{array}{rcl} 0  & \text{ for }  & x\le x(t_0) +R +\vartheta (t-t_0) \\
  1 &  \text{ for }  & x\ge  x(t_0) +R +\vartheta (t-t_0) +R^\theta+ (t_0-t)^{\theta}
 \end{array}
 \right. \; .
\end{equation}
 To simplify the notation, we set 
  $ \displaystyle y:=\frac{x-(x(t_0) +R +\vartheta (t-t_0))}{(R+\frac{1}{8}(t_0-t))^\theta}$. We  thus have 
  $\Psi(t,x) = \chi(y)$ and,   for any $ k\in \mathbb{N} $, 
\begin{equation}\label{deriv}
\frac{\partial^k}{\partial_x^k} \Psi (t,x)= \frac{1}{(R+\frac{1}{8}(t_0-t))^{k \theta}} \  \chi^{(k)}(y) \; .
\end{equation}
Moreover, for $ t\neq t_0$, 
\begin{align}
\partial_t \Psi (t,x)&=- \vartheta(R+\frac{1}{8}(t_0-t))^{-\theta} \chi'(y) + \frac{\theta \ y}{8(R+\frac{1}{8}(t_0-t))} \chi'(y) \nonumber \\
&= -\vartheta \Psi_x(t,x) + \frac{\theta \ y}{8(R+\frac{1}{8}(t_0-t))} \chi'(y)
.\label{psit}
\end{align}
Note also that,  for any $t\in\R $, $ \Psi(t)\in W^{3,\infty}(\R) $ and that the identity \eqref{deriv} with $ k=1 $ leads to 
$$
{\mathcal F}_x(\Psi_x(t,\cdot))(\xi) =e^{i(x(t_0)+R +\vartheta (t_0-t))\xi} \widehat{\chi'}((R+\frac{1}{8}(t_0-t))^\theta \, \xi) \ .
$$
Besides, we use the change of variable $\eta:= ((R+\frac{1}{8}(t_0-t))^\theta \, \xi$, which then gives $d \eta = ((R+\frac{1}{8}(t_0-t))^{\theta} d \xi$, and recall that $D_x^s \psi_x = \mathcal{F}_x^{-1} (|\xi|^s \ \hat{\psi}(\xi))$. Therefore, for $ s\ge 0 $, we get
$$
D^s_x \Psi_x(x) = (R+\frac{1}{8}(t_0-t))^{-\theta(s+1)}  \int_{\R} \displaystyle e^{i \frac{x-x(t_0)-R -\vartheta(t_0-t)}{(R+\frac{1}{8}(t_0-t))^\theta} \eta }
|\eta|^s \ \widehat{\chi'}(\eta) \, d \eta 
$$
which ensures that, for $ \delta\ge 1 $, the following holds 
\begin{equation}\label{deriv2}
\|D^{\delta}_x \Psi \|_{L^\infty_x} \lesssim (R+\frac{1}{8}(t_0-t))^{-\theta \delta} \; .
\end{equation}
In the sequel of this proof, $ C=C(\chi)  $ denotes a positive constant depending on $ \chi $ that can change from line to line. 
Differentiating $I^{+R}_{t_0}(t)$ with respect to time with \eqref{psit} in hand, we get for any $ t<t_0 $, 
\begin{eqnarray}\nonumber
\displaystyle \frac{d}{dt}I^{+R}_{t_0}(t) &=& \int_{\mathbb{R}} u_t \ u \ \Psi \ dx \ + \ \frac{1}{2} \int_{\mathbb{R}} u^2 \partial_t \Psi \ dx \nonumber \\
\displaystyle &=& \int u_{xx} (u_x \Psi + u \Psi_x) \ + \ \gamma \int \H u_x (u_x \Psi + u \Psi_x) \ - \ \frac{1}{3} \int \partial_x (u^3) \ \Psi \nonumber \\
\displaystyle && -\frac{\vartheta}{2} \int u^2 \ \Psi_x +\frac{1}{2(R+\frac{1}{8}(t_0-t))} \int \theta \ y \ u^2 \ \chi'(y) \nonumber \\
\displaystyle &=& -\frac{\vartheta}{2}  \int u^2 \ \Psi_x-\frac{3}{2} \int u_x^2 \Psi_x +\frac{1}{2} \int u^2 \Psi_{xxx} + \gamma \int \H u_x (u_x \Psi + u \Psi_x)  \nonumber \\
\displaystyle &&+ \frac{1}{3} \int u^3 \Psi_x +\frac{1}{2^4(R+\frac{1}{8}(t_0-t))} \int \theta \ y \ u^2 \ \chi'(y) \ dx .\label{main}
\end{eqnarray}
Firstly, it follows from  \eqref{deriv} that
 \begin{equation} \label{Est1}
 \frac{1}{2} \int u^2 \ \Psi_{xxx} \leq  \frac{1}{(R+\frac{1}{8}(t_0-t))^{3 \theta}} \|\chi^{'''}\|_{L^{\infty}} \int u^2 \leq C \  \frac{M(u)}{(R+\frac{1}{8}(t_0-t))^{3 \theta}} \ .
 \end{equation}
To treat terms involving the Hilbert transform, we will use commutator estimates.
For the first term, we notice that 
\begin{eqnarray*}
\displaystyle \int \H u_x \Psi u_x &=& - \int u_x \H(\Psi u_x) = - \int \H u_x \Psi u_x - \int u_x [\H, \Psi] u_x
\end{eqnarray*}
so that 
\begin{equation}\label{tgtg}
 \gamma \int \H u_x \Psi u_x = -\frac{\gamma}{2} \int u_x [\H, \Psi] u_x \ .
\end{equation}
On the other hand, integrating by parts and making use of \eqref{besov1} and \eqref{deriv2}, we get 
\begin{eqnarray} 
\bigg | \int u_x [\H, \Psi] u_x \bigg | =  \bigg | \int u \partial_x( [\H, \Psi] u_x) \bigg |  &\lesssim&  \|u(t)\|^2_{L^2_x} 
(\|D_x^{2-\varepsilon}\Psi(t)\|_{L^\infty_x}+\|D_x^{2+\varepsilon}\Psi(t)\|_{L^\infty_x}) \nonumber \\
& \lesssim& \frac{\|u(t)\|^2_{L^2_x}}{(R+\frac{1}{8}(t_0-t))^{(2-\varepsilon) \theta}}\;. \label{Est2}
 \end{eqnarray}
 To treat the second term, it  is convenient to introduce $ h= \sqrt{\partial_x \Psi} $ . Note that according to the definition of $ \chi $, $x\mapsto h(x) $ belongs to $ C^\infty(\R) $  and $\displaystyle \int_{\mathbb{R}} h^2 =1$. By Cauchy-Schwarz and Young's inequalities, the commutator estimate
 \eqref{comut}, \eqref{deriv2}  and the fact that $\H$ is unitary in $L^2$, we get
\begin{eqnarray}\nonumber
\bigg | \displaystyle \int \H u_x \Psi_x u\bigg | &=& \bigg | \int  \H u_x \ h^2 u\bigg | = \bigg |\int h \ u \ [h, \H] u_x + \int h \ u \ \H(h u_x)\bigg | \nonumber \\
\displaystyle &\leq & \|h\|_{L^\infty} \|u\|_{L^2} \|[h, \H] u_x\|_{L^2} + \frac{1}{4} \|h \ u\|^2_{L^2} + \|\H(h \ u_x) \|^2_{L^2} \nonumber \\
\displaystyle &\leq & \|h\|_{L^\infty} \|u\|_{L^2} \| h_x\|_{L^\infty} \|u\|_{L^2} + \frac{1}{4} \|h \ u\|^2_{L^2} +  \|h \ u_x \|^2_{L^2} \nonumber \\
\displaystyle &\le& C\, \frac{\|u\|_{L^2}^2}{(R+\frac{1}{8}(t_0-t))^{2\theta}}  + \frac{1}{4} \int h^2 u^2 \ +  \int h^2 u_x^2 \ .\nonumber \end{eqnarray}
Thus, 
\begin{equation}\label{Est3}
\bigg |\gamma \int \H u_x \Psi_x u \bigg |\leq C\, |\gamma| \frac{\|u\|^2_{L^2}}{(R+\frac{1}{8}(t_0-t))^{2\theta}} + \frac{|\gamma|}{4} \int u^2 \ \Psi_x  + |\gamma|\int u^2_x \ \Psi_x .
\end{equation}
Now, in view of \eqref{Psit}, for $ t\le t_0$,  it holds 
\begin{equation}\label{support}
\text{supp}\,  \Psi_x(t,.) \subset [x(t_0)+R+\vartheta(t-t_0),+\infty[ \subset [x(t)+R,+\infty[ 
\end{equation}
since $ \dot{x}\ge 5/6\ge 3/4 $. 
Therefore, according to \eqref{loc}, for $R\ge R_0$ we have 
\begin{equation}\label{est4}
\frac{1}{3} \int u^3 \ \Psi_x \leq \frac{1}{3} \|u\|_{L^{\infty}(supp \Psi_x)} \int u^2 \ \Psi_x \leq 2^{-6} \int u^2 \ \Psi_x dx \ .
\end{equation}
Finally, for the last term in \eqref{main}, we use Young's inequality to get
\begin{equation}\label{young}
 \frac{\theta |y|}{(R+\frac{1}{8}(t_0-t))} \leq \frac{2^{-2}}{(R+\frac{1}{8}(t_0-t))^{\theta}} + \frac{ \theta^2 \ y^2}{(R+\frac{1}{8}(t_0-t))^{2-\theta}} 
\end{equation}
and thus 
$$\bigg | \int \frac{\theta \ y \ u^2 \ \chi'(y)}{2^4(R+\frac{1}{8}(t_0-t))} \ dx\bigg | \leq 2^{-6} \int \frac{\chi'(y)}{(R+\frac{1}{8}(t_0-t))^{\theta}} \ u^2 \ dx \ + \ \frac{2^{-4}\theta^2}{(R+\frac{1}{8}(t_0-t))^{2-\theta}} \int y^2 \ \chi'(y) \ u^2 \ dx .
$$
Recalling that  supp $\chi' \subset [0,1]$ and so  $\| y^2 \ \chi'\|_{L^\infty} \lesssim 1$, we eventually obtain
\begin{eqnarray} 
\bigg |\displaystyle \int \frac{\theta \ y}{2(R+\frac{1}{8}(t_0-t))} \ u^2 \ \chi'(y) \ dx\bigg | &\leq & 2^{-6} \int u^2 \ \Psi_x \ dx + \frac{C \ \|u\|^2_{L^2}}{(R+\frac{1}{8}(t_0-t))^{2-\theta}} \ . \quad \quad \label{est5}
\end{eqnarray}
 Gathering \eqref{Est1}, \eqref{Est2}, \eqref{Est3}, \eqref{est4} and \eqref{est5}, we end up with
\begin{eqnarray}
\displaystyle \frac{d}{dt} I^{+R}_{t_0}(t) &\leq & \frac{4 |\gamma| -6}{4} \int u_x^2(t) \ \Psi_x(t) \ dx \ + \bigg ( \frac{|\gamma|}{4} +2^{-6}+2^{-6} -\frac{\vartheta}{2}  \bigg) \int u^2(t) \ \Psi_x(t) \ dx \nonumber \\
\displaystyle & + & C\, M(u(t))\, \Bigl(  (R+\frac{1}{8}(t_0-t))^{-(2-\varepsilon) \theta} +  (R+\frac{1}{8}(t_0-t))^{ \theta-2}\Bigr) \nonumber \\ 
\nonumber \\
\displaystyle &\leq & -\int u_x^2(t) \Psi_x(t) \ dx -2^{-5} \int u^2(t) \Psi_x(t) \ dx \nonumber \\
\displaystyle & &+ C\, M(u(t))\,\Bigl(  (R+\frac{1}{8}(t_0-t))^{-(2-\varepsilon) \theta} +  (R+\frac{1}{8}(t_0-t))^{ \theta-2}\Bigr)
\label{finalder}
\end{eqnarray}
where we used that $R+\frac{1}{8}(t_0-t)\ge R_0\ge 1 $ and $ |\gamma|\le 1/2$.
Integrating this differential inequality on $ ]t,t_0[ $ for $ t<t_0 $, we eventually get 
$$
 I^{+R}_{t_0}(t_0) - I^{+R}_{t_0}(t) \le  C\, M(u)\, (R^{1-(2-\varepsilon) \theta}+R^{\theta-1}) \; .
$$
Clearly, the decay in the above relation  is optimal for $ \theta = \frac{2}{3}+ $, while for $ \theta= \frac{3}{4} $ it  leads to \eqref{mono} for $ R>R_0\ge 1 $. The case $ 0<R\le R_0 $ then follows directly from the conservation of the $ L^2$-norm. 

Analogously, the proof of  \eqref{mono2} follows the same lines by replacing $ \Psi $ by $  \tilde{\Psi} $ defined by 
 \begin{equation}\label{defPsi2}
 \tilde{\Psi}(t,x) = \chi \Bigl(\frac{\cdot-(x(t_0) -2 R +\vartheta (t-t_0))}{(R+\frac{1}{8}(t_0-t))^\frac{3}{4}}\Bigr) .
 \end{equation}
 This time for $ t\ge t_0 $, it holds 
 \begin{equation}\label{Psit2}
\tilde{\Psi}(t,x)=\left\{ \begin{array}{rcl} 0  & \text{ for }  & x\le x(t_0) -2R +\vartheta (t-t_0) \\
  1 &  \text{ for }  & x\ge  x(t_0) -2R +\vartheta (t-t_0) + R^\frac{3}{4} + (t_0-t)^\frac{3}{4}
 \end{array}
 \right. \; .
\end{equation}
 Since for $ R >0 $ large enough, $ \tau \mapsto -R+R^\frac{3}{4} + \tau^\frac{3}{4}-\frac{1}{8} \tau$  is  clearly negative on $\R_+ $, we deduce that there exists $ \tilde{R}_0 \ge R_0 $ such that for all $ R\ge  \tilde{R}_0$, 
 $$
  \tilde{\Psi}(t,x)=1 \quad \text{for} \; t\ge t_0 \quad \text{and} \quad   x\ge  x(t_0) -R +(\vartheta+\frac{1}{8}) (t-t_0) . 
  $$
   In particular, as $ \dot{x}(t) \ge 5/6 > \vartheta+\frac{1}{8}$, it follows that 
   \begin{equation}\label{dodo}
    \tilde{\Psi}_x(t,x)=0 \quad \text{for} \; t\ge t_0 \quad \text{and} \quad   x\ge  x(t) -\tilde{R}_0 \; .
   \end{equation}
   According to \eqref{loc}, we recover that, as soon as $ R\ge \tilde{R}_0$, $|u(t,x)| < 2^{-6} $ for any $ x\in \text{supp}\,  \tilde{\Psi}_x(t) $ but this time for $ t\ge t_0$. Then, exactly the same calculations as above lead to  \eqref{mono2} which completes the proof of the proposition.
 \end{proof}
\section{Almost monotonicity for an energy at the right of a mainly localized solution}
In this section, we prove the almost monotonicity of the localized quantity $ 4M +E $ at the right of a mainly localized solution. We follow the same kind of arguments as in the preceding section but with more technical difficulties.
\begin{proposition} \label{H1Monot} 
(Almost Monotonicity at the $ H^1$-level)\\
%Let $0 <\lambda <1$. 
For a solution $ u\in C(\R;H^1(\R)) $ to \eqref{MainEq}, $t, \ t_0 \in \mathbb{R}$ and $\vartheta \in [\frac{3}{8},\frac{5}{8}] $, we set 
\begin{equation}\label{deftEnergy}
J_{t_0}^{+R}(t) :=   \int  \Bigl(\frac{1}{2} u_x^2 \ - \frac{\gamma}{2} u \H u_x \  - \frac{1}{6} \ u^3 \Bigr)(t) \ 
\chi \Bigl(\frac{\cdot-(x(t_0) + R +\vartheta(t-t_0))}{(R+\frac{1}{8}|t_0-t|)^{3/4}}\Bigr) 
\end{equation}
and 
\begin{equation}\label{deftEnergy2}
J_{t_0}^{-2R}(t) :=   \int  \Bigl(\frac{1}{2} u_x^2 \ - \frac{\gamma}{2} u \H u_x \  - \frac{1}{6} \ u^3 \Bigr)(t) \ 
\chi \Bigl(\frac{\cdot-(x(t_0) -2 R +\vartheta(t-t_0))}{(R+\frac{1}{8}|t_0-t|)^{3/4}}\Bigr) \ .
\end{equation}
For $|\gamma|\le 1/2$ and under the same assumptions as in Proposition \ref{L2Monot}, there exist $C>0$ only depending on $ M(u_0)$, $E(u_0) $ and $ R_0 $  such that for all $R>0$ and $t\le  t_0$,
\begin{equation}
4 I^{+R}_{t_0}(t_0)+J^{+R}_{t_0}(t_0) \leq  4 I^{+R}_{t_0}(t) +J^{+R}_{t_0}(t)  +C \ R^{-\frac{1}{4}} \sup_{t \in \mathbb{R}} \|u(t)\|^2_{H^1} (1+\|u(t)\|^2_{H^1}) \label{monotH1}
\end{equation}
whereas 
for all $ R>0 $ and $ t\ge t_0$ 
\begin{equation}
4 I^{-2R}_{t_0}(t_0)+J^{-2R}_{t_0}(t_0) \leq  4 I^{-2R}_{t_0}(t) +J^{-2R}_{t_0}(t)  +C \ R^{-\frac{1}{4}} \sup_{t \in \mathbb{R}} \|u(t)\|^2_{H^1} (1+\|u(t)\|^2_{H^1})\ .  \label{monotH12}
\end{equation}
\end{proposition}

\begin{proof} We first prove \eqref{monotH1}. Consider $\Psi $ as that defined in \eqref{defPsi}.  
 We compute 
\begin{eqnarray} 
\frac{d}{dt} J^{+R}_{t_0}(t) & =&  \frac{1}{2}  \frac{d}{dt}\int \Psi \ u_x^2 \ - \frac{\gamma}{2} \frac{d}{dt}\int \Psi \ u \H u_x \  - \frac{1}{6} \frac{d}{dt}\int \Psi \ u^3 \nonumber \\
&:= &  \Theta_1 -\frac{\gamma}{2} \Theta_2 + \Theta_3 \ . \label{noH1}
\end{eqnarray} 
We are going to develop these contributions one by one 
starting with 
$$ \Theta_1 = \frac{1}{2} \frac{d}{dt} \int \Psi \ u_x^2
 = \int \Psi \ u_x \ u_{xt} \ + \frac{1}{2} \int \Psi_t \ u_x^2 := B + C \; .$$
By integration by parts, we  get 
\begin{eqnarray}\nonumber
\displaystyle B &=&  \int \Psi \ u_x \ \partial_x (- u_{3x} - \gamma \H u_{2x} - u \ u_x) \nonumber \\
\displaystyle &=& - \int \Psi u_x u_{4x} \ - \gamma \int \Psi u_x \H u_{3x} \ - \int \Psi u_x ( u \ u_x)_x \nonumber \\
\displaystyle & = & \int \Psi_x u_x u_{3x} \ + \int \Psi u_{2x} u_{3x} \ + \gamma \int \Psi_x u_x \H u_{2x} \ + \gamma \int \Psi u_{2x} \H u_{2x} \ - \int \Psi u_x (u \ u_x)_x \nonumber \\
\displaystyle &=& - \int \Psi_{2x} u_x u_{2x}  \ -\int \Psi_x u^2_{2x} \ + \frac{1}{2} \int \Psi \partial_x (u_{2x}^2) \ + \gamma \int \Psi_x u_x \H u_{2x} \ + \gamma \int \Psi u_{2x} \H u_{2x} \nonumber \\
\displaystyle & & - \int \Psi u_x (u \ u_x)_x \nonumber \\
\displaystyle &=& \frac{1}{2} \int \Psi_{3x} u_x^2 \ - \frac{3}{2} \int \Psi_x u^2_{2x} \ + \gamma \int \Psi_x u_x \H u_{2x} \ + \gamma \int \Psi u_{2x} \H u_{2x} \ - \int \Psi u_x (u  u_x)_x \, \nonumber
\end{eqnarray}
where we notice that the first term in the right-hand side of the last inequality can be estimated thanks to \eqref{deriv} by 
$$ \frac{1}{2} \bigg | \int \Psi_{3x}u_x^2 \bigg | \leq \frac{1}{2} \|\Psi_{3x}\|_{L^\infty} \ \|u_x\|^2_{L^2} \leq \frac{c_1}{(R +\frac{1}{8} (t_0 -t))^{3 \theta}} \|u\|^2_{H^1} .$$
On the other hand, according to \eqref{psit} and  \eqref{young}, we can see that
\begin{eqnarray}\nonumber 
\displaystyle C &=& \frac{1}{2} \int \Psi_t u_x^2 = \frac{1}{2} \bigg [ -\vartheta\int \Psi_x u_x^2 \ +\frac{1}{8}\int \frac{\theta y}{R+\frac{1}{8}(t_0 -t)} \chi'(y) u_x^2 \bigg ] \nonumber \\
\displaystyle & \leq & (-\frac{\vartheta}{2}+2^{-6})\int \Psi_x u_x^2 +\frac{1}{8} \frac{2 \theta^2}{(R+\frac{1}{8}(t_0 -t))^{2-\theta}} \int y^2 \chi'(y) u_x^2 \bigg] \nonumber \\
\displaystyle & \leq &  (-\frac{\vartheta}{2}+2^{-6}) \int \Psi_xu_x^2 \ +\frac{c_2}{(R+\frac{1}{8}(t_0 -t))^{2-\theta}} \|u_x\|^2_{L^2}  \ . \nonumber
\end{eqnarray}
Eventually, the first term $\Theta_1$  thus  satisfies 
\begin{align}
\displaystyle \Theta_1  \le& - \frac{3}{2}   \int \Psi_x u^2_{2x}-(\frac{\vartheta}{2}-2^{-6}) \int \Psi_x u_x^2  + \gamma \int \Psi_x u_x \H u_{2x} \ + \gamma \int \Psi u_{2x} \H u_{2x} \nonumber  \\
\displaystyle &- \int \Psi u_x (u u_x)_x+ \bigg [ \frac{c_1}{(R + \frac{1}{8}(t_0 -t))^{3 \theta}} + \frac{c_2}{(R+\frac{1}{8}(t_0 -t))^{2-\theta}} \bigg] \|u\|^2_{H^1} \ . \quad \quad \label{1-H1}
\end{align}
We now move  on to 
\begin{eqnarray*}\nonumber
\displaystyle \Theta_2 = \frac{d}{dt} \int \Psi u \H u_x & = & \int \Psi_t u  \H u_x \ + \int \Psi u_t \H u_x \ + \int \H(\Psi u)_x u_t \nonumber \\
\displaystyle &:=& D + E + F \ .\nonumber
\end{eqnarray*}
Using \eqref{psit} and \eqref{young} first, then  the commutator estimate \eqref{comut} and the fact that $\H$ is $L^2$-unitary along with Young's inequalities, $D$ can be controlled by
\begin{eqnarray} 
\displaystyle |D|  &=& \bigg|-\frac{3}{4}\int \Psi_x u \H u_x \quad + \int \frac{\theta y }{R+\frac{1}{8}(t_0-t)} \chi'(y) u \H u_x \bigg| \nonumber \\
\displaystyle & \leq &\bigg|  \int \Psi_x u \H u_x \bigg|  \quad + \ \frac{c_2 \ \|u\|_{L^2} \|u_x\|_{L^2}}{(R+\frac{1}{8}(t_0-t))^{2-\theta}} \nonumber \\
\displaystyle & \leq & \bigg|  \int \sqrt{\Psi_x} u [\H, \sqrt{\Psi_x}] u_x \ + \ \int \sqrt{\Psi_x} u \H (\sqrt{\Psi_x} u_x)  \bigg| \ + \ \frac{c_2 \ \|u\|^2_{H^1}}{(R+\frac{1}{8}(t_0-t))^{2-\theta}} \nonumber \\
\displaystyle & \leq & C\,   \|\sqrt{\Psi_x}\|_{L^\infty} \|(\sqrt{\Psi_x})'\|_{L^\infty} \|u\|^2_{L^2} + \frac{1}{8} \int \Psi_x u_x^2 \ + \ 2 \int \Psi_x u^2 \ + \ \frac{c_2 \ \|u\|^2_{H^1}}{(R+\frac{1}{8}(t_0-t))^{2-\theta}} \nonumber \\
\displaystyle & \leq & \frac{1}{8} \int \Psi_x u_x^2 \ + \ 2 \int \Psi_x u^2 \ + \ \frac{c_2 \ \|u\|^2_{H^1}}{(R+\frac{1}{8}(t_0-t))^{2-\theta}}+ \frac{c_3 \ \|u\|^2_{L^2}}{(R+\frac{1}{8}(t_0-t))^{2\theta}} \; .\label{estD} 
\end{eqnarray}
Whereas, the second term  $E$ contributing in $\Theta_2$ can be developed as
\begin{eqnarray}
\displaystyle E &=& \int  \Psi(- u_{3x} -\gamma \H u_{2x} - u u_x) \H u_x \nonumber \\
\displaystyle & =&- \int \Psi u_{3x} \H u_x \ - \gamma  \int \Psi  \H u_{2x} \H u_x - \int \Psi u u_x \H u_x  \nonumber \\
\displaystyle &:=& E_1 + \gamma E_2 - \int \Psi u u_x \H u_x \ .\label{E1}
\end{eqnarray}
Firstly, 
\begin{eqnarray}
\displaystyle E_1 &=& \int \Psi u_{2x} \H u_{2x} \ - \int \H(\Psi_x u_{2x}) \,  u_x \nonumber \\
\displaystyle &= &  \int \Psi u_{2x} \H u_{2x} \ - \int \H (\sqrt{\Psi_x} (\sqrt{\Psi_x} u_x)_x) u_x \ + \int \H (\sqrt{\Psi_x} (\sqrt{\Psi_x})_x u_x) u_x \nonumber \\
\displaystyle &=& \int \Psi u_{2x} \H u_{2x} \ - \int \H (\sqrt{\Psi_x} (\sqrt{\Psi_x} u_x)_x) u_x \ + \int \H (\sqrt{\Psi_x} (\sqrt{\Psi_x})_x u_x) u_x \nonumber \\
\displaystyle &:=& \int \Psi u_{2x} \H u_{2x} + E_1^2 + E_1^3 \ .\label{E2}
\end{eqnarray}
We rewrite $ E_1^2 $ in the following way
\begin{eqnarray}\nonumber
\displaystyle E_1^2 &=& - \int \H (\sqrt{\Psi_x} (\sqrt{\Psi_x} u_x)_x) u_x \nonumber \\
\displaystyle &=& -\int [\H, \sqrt{\Psi_x}] (\sqrt{\Psi_x} u_x)_x u_x \ + \int \H (\sqrt{\Psi_x} u_x )_x \sqrt{\Psi_x} u_x \nonumber \\
\displaystyle & = & -\int [\H, \sqrt{\Psi_x}] (\sqrt{\Psi_x} u_x)_x u_x \ - \int \H (\sqrt{\Psi_x} u_x ) (\sqrt{\Psi_x} u_x)_x \nonumber \\
\displaystyle & = & -\int [\H, \sqrt{\Psi_x}] (\sqrt{\Psi_x} u_x)_x u_x \ -  \int \H (\sqrt{\Psi_x} u_x ) (\sqrt{\Psi_x} u_{2x}) - \int \H (\sqrt{\Psi_x} u_x ) ((\sqrt{\Psi_x})_x u_x) \ . \nonumber
\end{eqnarray}
Hence,  by the commutator estimate \eqref{comut}, \eqref{deriv2},  Holder's and Young's Inequalities, we get 
\begin{eqnarray}\nonumber
\displaystyle |E_1^2| & \leq & C\, \| (\sqrt{\Psi_x})_x\|_{L^\infty} \|\sqrt{\Psi_x} u_x\|_{L^2} \|u_x\|_{L^2} +  \|\sqrt{\Psi_x} u_x\|_{L^2}  \|\sqrt{\Psi_x} u_{2x}\|_{L^2}
\nonumber \\
& &  +  \displaystyle \|\sqrt{\Psi_x} u_x\|_{L^2} \ \|(\sqrt{\Psi_x})_x u_x\|_{L^2} \nonumber \\
\displaystyle & \leq &  \int \Psi_x u^2_{2x} +  \frac{1}{4}\int \Psi_x u_x^2 \ +
 (C+1) \|\sqrt{\Psi_x}\|_{L^2}  \|(\sqrt{\Psi_x})_x\|_{L^\infty} \|u_x\|_{L^2}^2  \nonumber \\
\displaystyle & \leq & \int \Psi_x u^2_{2x} +  \frac{1}{4} \int \Psi_x u_x^2 \ + \frac{c_4 \ \|u_x\|^2_{L^2}}{(R+\frac{1}{8}(t_0-t))^{2\theta}} \ . \label{E3}
\end{eqnarray}
In the same way, 
\begin{eqnarray}\nonumber
\displaystyle |E_1^3| & \leq & \|\sqrt{\Psi_x} (\sqrt{\Psi_x})_x u_x \|_{L^2} \ \|u_x\|_{L^2} \nonumber \\
 &\leq & \|\sqrt{\Psi_x}\|_{L^2}  \|(\sqrt{\Psi_x})_x\|_{L^\infty} \|u_x\|_{L^2}^2\le  \frac{c_4}{(R+\frac{1}{8}(t_0-t))^{2\theta}} \|u_x\|^2_{L^2} \ .\label{E4}
 \end{eqnarray}
Now, using  the fact that $ \H^2 = -Id$, we get
\begin{eqnarray*}\nonumber
\displaystyle E_2 &=& \int \H(\Psi \H u_{2x}) u_x = \int \H ( [\Psi,\H] u_{2x} )  u_{x}-\int \Psi u_{2x} u_x \\
& = &  \int \partial_x ([\Psi, \H] u_{2x}) \H u +\frac{1}{2} \int \Psi_x u_x^2 
\end{eqnarray*}
whereas the commutator estimate \eqref{besov1} together with  \eqref{deriv2} leads to 
\begin{eqnarray}
\displaystyle |E_2| &\le & \frac{1}{2} \int \Psi_x u_x^2 + \frac{c_5}{(R+\frac{1}{8}(t_0-t)^{})^{(2-\varepsilon)\theta}} \|u_x\|_{L^2}\|u\|_{L^2} \; . \label{E5}
\end{eqnarray}
Gathering \eqref{E1}-\eqref{E5}, we thus infer that 
\begin{align}
\displaystyle E =\int \Psi u_{2x} \H u_{2x} - \int \Psi u u_x \H u_x + \Lambda_1(u) \label{EstE}
\end{align}
where, recalling that $ R+\frac{1}{8}(t_0-t)\ge 1 $, it holds
\begin{align*}
 |\Lambda_1(u)| & \le \ \int \Psi_x u^2_{2x} + (  \frac{1}{4} +\frac{|\gamma|}{2}) \int \Psi_x u_x^2   +   \frac{c \ \|u\|^2_{H^1}}{(R+\frac{1}{8}(t_0-t))^{2\theta-\varepsilon}} \, .
\end{align*}
We still have to study the contribution of $F$;
% Using the fact that $\H^2 = -Id$, we obtain
\begin{eqnarray}\nonumber
\displaystyle F &=&  \int \H(\Psi \ u)_x (- u_{3x} - \gamma \H u_{2x} - u u_x) \nonumber \\
\displaystyle &=& - \int (\Psi u)_x \H u_{3x} \ - \gamma \int (\Psi u)_x \H^2 u_{2x} \ - \int \H (\Psi u)_x u u_x \nonumber \\
\displaystyle &=& \int (\Psi u)_{2x} \H u_{2x} \ + \gamma \int (\Psi u)_x u_{2x} \ -\int  \H (\Psi u)_x u u_x \nonumber \\
\displaystyle &:=& F_1 + \gamma F_2 + F_3 \ .\nonumber
\end{eqnarray}
We decompose $F_1$ as
\begin{eqnarray}\nonumber
\displaystyle F_1 &=& \int \Psi u_{2x} \H u_{2x} \ + 2 \int \Psi_x u_x \H u_{2x} \ + \int \Psi_{2x}\,  u \H u_{2x} \ ,
\end{eqnarray}
and notice that  the last term in the above right-hand side can be estimated as follows
\begin{eqnarray}\nonumber
\displaystyle \Bigl|\int \Psi_{2x}\,  u \H u_{2x}\Bigr| &=& \bigg | - \int \Psi_{3x} u \H u_x \ - \int \Psi_{2x} u_x \H u_x \bigg | \nonumber \\
\displaystyle & \leq & ( \|\Psi_{3x}\|_{L^\infty} + \|\Psi_{2x} \|_{L^\infty} ) \|u\|^2_{H^1} \nonumber \\
\displaystyle & \leq & \frac{c_1 \ \|u\|^2_{H^1}}{(R + (t_0 -t))^{3 \theta}} \ + \frac{c_5 \ \|u\|^2_{H^1}}{(R + (t_0 -t))^{2 \theta}} \ .\nonumber
\end{eqnarray}
On the other hand, $ F_2 $ shall be decomposed as 
\begin{eqnarray}\nonumber
\displaystyle F_2 &=& \int (\Psi u)_x u_{2x} = \int \Psi_x u u_{2x} \ + \int \Psi u_x u_{2x} \nonumber \\
\displaystyle &=& - \int \Psi_{2x} u u_x - \int \Psi_x u_x^2 \ -\frac{1}{2} \int \Psi_x u_x^2 \nonumber \\
\displaystyle &=& \frac{1}{2} \int \Psi_{3x} u^2 \ -\frac{3}{2} \int \Psi_x u_x^2  \nonumber
\end{eqnarray}
which leads to 
\begin{eqnarray}\nonumber
\displaystyle |\gamma F_2| &\leq & \frac{3 |\gamma|}{2} \int \Psi_x u_x^2 \ + \frac{|\gamma|}{2} \ \|\Psi_{3x}\|_{L^\infty} \ \|u\|^2_{L^2} \ .\nonumber
\end{eqnarray}
Hence, $ F $ satisfies 
\begin{align}
\displaystyle F =  \int \Psi u_{2x} \H u_{2x}  + 2 \int \Psi_x u_x \H u_{2x}  -\int  \H (\Psi u)_x u u_x + \Lambda_2(u) 
\label{estF}
\end{align}
with 
\begin{align*}
|\Lambda_2(u)|& \le   \frac{3 |\gamma|}{2} \int \Psi_x u_x^2  + \frac{c \ \|u\|^2_{H^1}}{(R + (t_0 -t))^{2 \theta}}  \; .
\end{align*}
Shortly, the contribution of the second term in \eqref{noH1} can be seen by gathering \eqref{estD}, \eqref{EstE}, and \eqref{estF} as
\begin{align}
\displaystyle -\frac{\gamma}{2} \Theta_2 
  \le  & \frac{\gamma}{2} \int  \H (\Psi u)_x u u_x \ -\frac{\gamma}{2} \int \H(\Psi u u_x) u_x \ -\gamma \int \Psi u_{2x} \H u_{2x} \ -\gamma \int \Psi_x u_x \H u_{2x} \nonumber \\
 & + |\gamma|  \int \Psi_x u^2_{2x} +3 |\gamma| \int \Psi_x u_x^2 +  \frac{c \ \|u\|^2_{H^1}}{(R+\frac{1}{8}(t_0-t))^{2-\theta}}  +  \frac{c \ \|u\|^2_{H^1}}{(R+\frac{1}{8}(t_0-t))^{(2-\varepsilon)\theta}} \, \label{2-H1}
\end{align}
Lastly, we tackle $\Theta_3$ which we develop as 
\begin{eqnarray}
 \Theta_3&=& -\frac{1}{6} \frac{d}{dt} \int \Psi u^3 = -\frac{1}{6} \int \Psi_t u^3 \ -\frac{1}{2} \int \Psi u^2 u_t := G + H \ . \nonumber
\end{eqnarray}
By  \eqref{psit}, \eqref{young} and \eqref{loc}, $G$ can be  controlled for $ R\ge R_0 $  in the following way
\begin{eqnarray}
\displaystyle |G|& = & \frac{1}{6} \bigg | -\frac{3}{4} \int \Psi_x u^3 \ + \int \frac{\theta \ y}{(R+\frac{1}{8}(t_0-t))} \chi'(y) \ u^3 \bigg | \nonumber \\
\displaystyle & \leq & \frac{1}{4}  \ \|u\|_{L^\infty (\text{supp}\,  \Psi_x(t))} \int \Psi_x u^2 \ + \frac{c_2}{(R+\frac{1}{8}(t_0-t))^{2-\theta}} \|u\|_{L^\infty (\text{supp}\, \Psi_x(t))}\ \|u\|^2_{L^2}  \nonumber \\
\displaystyle & \leq & 2^{-8}  \int \Psi_x u^2+ \frac{c}{(R+\frac{1}{8}(t_0-t))^{2-\theta}}  \|u\|^2_{L^2} \ , 
\label{estG}
\end{eqnarray}
where in the last step we make use of \eqref{support}.\\
Last in order, we proceed to estimate the final term $H$. 
\begin{eqnarray}\nonumber
\displaystyle H &=& \frac{1}{2} \int \Psi u^2 u_{3x} \ + \frac{\gamma}{2} \int \Psi u^2 \H u_{2x} \ + \frac{1}{2} \int \Psi u^3 u_x := H_1 + H_2 + H_3 \ ,\nonumber
\end{eqnarray}
where, integrating by parts and  making use of  \eqref{deriv2}, \eqref{support}  and \eqref{loc} we get for $ R\ge R_0 $  that
\begin{eqnarray}\nonumber
\displaystyle H_1 &=& - \frac{1}{2} \int \Psi_x u^2 u_{2x} \ - \int \Psi u u_x u_{2x} = \frac{1}{2} \int \Psi_{2x} u^2 u_x \ + \int \Psi_x u u_x^2 \ + \frac{1}{2} \int \Psi_x u u_x^2 \ + \frac{1}{2} \int \Psi u_x^3 \nonumber \\
\displaystyle & =& \frac{1}{2} \int \Psi_{2x} u^2 u_x \ + \frac{3}{2} \int \Psi_x u u_x^2 \ + \frac{1}{2} \int \Psi u_x^3 \nonumber \\
\displaystyle &\leq & \frac{1}{2} \|\Psi_{2x}\|_{L^\infty} \|u\|_{L^\infty}  \|u\|_{L^2} \|u\|_{H^1} \ + \frac{3}{2}  \|u\|_{L^\infty (\text{supp}\,  \Psi_x(t))} \int \Psi_x u_x^2 \ + \frac{1}{2} \int \Psi u_x^3  \nonumber\\
\displaystyle &\leq & \frac{1}{2} \int \Psi u_x^3+ 2^{-5}  \int \Psi_x u_x^2 \ +  \frac{c\, \|u\|_{H^1}^3}{(R+\frac{1}{8}(t_0-t))^{2\theta}} \label{estH1}
\; .
\end{eqnarray}
On the other hand,  \eqref{comut}, \eqref{support}  and \eqref{loc} lead to
\begin{eqnarray}\nonumber
\displaystyle H_2 &=& -\frac{\gamma}{2} \int \Psi_x u^2 \H u_x \ - \gamma \int \Psi u u_x \H u_x \nonumber \\
&=& -\frac{\gamma}{2} \int \sqrt{\Psi_x}  u^2 \, [\sqrt{\Psi_x},\H] u_x)  -\frac{\gamma}{2} \int \sqrt{\Psi_x}  u^2 \H (\sqrt{\Psi_x} u_x) \ - \gamma \int \Psi u u_x \H u_x \nonumber \\
\displaystyle & \leq &8 |\gamma| \int \Psi_x u^4 +\frac{|\gamma|}{2^4}  \int \Psi_x u_x^2+ \|[\sqrt{\Psi_x},\H] u_x\|_{L^2}^2 +\gamma \int \H (\Psi u u_x) u_x  \nonumber \\
\displaystyle & \leq & \gamma \int \H (\Psi u u_x) u_x+\frac{|\gamma|}{2^4} \int \Psi_x u_x^2+8 |\gamma| \|u\|_{L^\infty(\text{supp} \Psi'(t))}^2 \int \Psi_x u^2 + \|\partial_x(\sqrt{\Psi_x})\|_{L^\infty}^2 \|u\|_{L^2}^2 \nonumber \\
\displaystyle & \leq & \gamma \int \H (\Psi u u_x) u_x+\frac{|\gamma|}{2^4} \int \Psi_x u_x^2+2^{-8} \int \Psi_x u^2 + \frac{c\, \|u\|_{L^2}}{(R+\frac{1}{8}(t_0-t))^{3\theta}}\label{estH2}
\end{eqnarray}
while 
$$ H_3 = \int \Psi u^3 u_x = - \frac{1}{4} \int \Psi_x u^4 \quad \leq 0 \ .$$
Therefore, we end up with
\begin{eqnarray} \nonumber
\displaystyle H & \leq &  \frac{1}{2} \int \Psi u_x^3 +\gamma \int \H (\Psi u u_x) u_x +   \frac{1}{6} \int \Psi_x u_x^2 + 2^{-8}  \int \Psi_x u^2 
+ \frac{c\, \|u\|_{H^1}^2
(1+\|u\|_{H^1})}{(R+\frac{1}{8}(t_0-t))^{2\theta}} \ . \\
\label{estH}
\end{eqnarray}
In consequence, combining \eqref{estG} and \eqref{estH}, $\Theta_3$ is estimated by
\begin{eqnarray}
\displaystyle \Theta_3 &\leq & 
  \frac{1}{2} \int \Psi u_x^3 +\gamma \int \H (\Psi u u_x) u_x +   \frac{1}{6} \int \Psi_x u_x^2 + 2^{-6}  \int \Psi_x u^2 \nonumber \\
& & + \frac{c\, \|u\|_{H^1}^2
(1+\|u\|_{H^1})}{(R+\frac{1}{8}(t_0-t))^{2\theta}} + \frac{c\, \|u\|_{L^2}^2}{(R+\frac{1}{8}(t_0-t))^{2-\theta}} \ . \label{3-H1}
\end{eqnarray}
Drawing things to a close, we look at \eqref{1-H1}, \eqref{2-H1}, and \eqref{3-H1}; \eqref{noH1} now gives
\begin{eqnarray}
\frac{d}{dt} J^{+R}_{t_0}(t) & =&\Theta_1 -\frac{\gamma}{2} \Theta_2 + \Theta_3\nonumber \\
 &\leq & \frac{\gamma}{2} \bigg [ \int \H(\Psi u)_x u u_x \ + \int \H (\Psi u u_x) u_x \bigg ] + \frac{1}{2} \int \Psi u_x^3 \ - \int \Psi u_x (u u_x)_x \nonumber \\
\displaystyle &&+2^{-6} \int \Psi_x u^2 -(\frac{\vartheta}{2}-\frac{1}{6}-2^{-6}-3|\gamma|) \int \Psi_x u_x^2 -  (\frac{3}{2}-|\gamma|) \int \Psi_x u_{2x}^2 \nonumber \\
\displaystyle &&+ \frac{c\|u\|_{H^1}^2
(1+\|u\|_{H^1})}{(R+ (t_0 -t))^{(2-\varepsilon) \theta}} \ + \frac{c \ \|u\|^2_{H^1}}{(R+\frac{1}{8}(t_0-t))^{2-\theta}} \; .\label{finH1}
\end{eqnarray}
Regarding the first line of \eqref{finH1}, we start as follows
\begin{eqnarray} 
 \frac{\gamma}{2} \int \H(\Psi u)_x u u_x &+& \frac{\gamma}{2} \int \H (\Psi u u_x) u_x  = \frac{\gamma}{2} \int \H  (\Psi_x u+\Psi u_x) u u_x  -\frac{\gamma}{2} \int \Psi u u_x \H u_x \nonumber \\
 &=&  \frac{\gamma}{2} \int   u u_x \, [\H\partial_x  ,\Psi] u  +\frac{\gamma}{2}  \int   \H (\Psi_x u) u u_x  \nonumber \\
  &=&  -\frac{\gamma}{4} \int   u^2 \, \partial_x([\H\partial_x  ,\Psi] u)  -\frac{\gamma}{4}  \int   \H (\Psi_{2x} u) u^2
     -\frac{\gamma}{4}  \int   \H (\Psi_{x} u_x) u^2 \nonumber \\
 &=&  -\frac{\gamma}{4} \int   u^2 \, \partial_x([\H\partial_x  ,\Psi] u)  +\frac{\gamma}{4}  \int   \Psi_{2x}\,  u  \H(u^2)
     -\frac{\gamma}{4}  \int  [\H,\sqrt{\Psi_x}] ( \sqrt{\Psi_{x}}  u_x) \, u^2 \nonumber \\
     & & -\frac{\gamma}{4}  \int \H( \sqrt{\Psi_{x}} u_x) \sqrt{\Psi_x}u^2\nonumber \\
     &: =& \frac{\gamma}{4}(A_1+A_2+A_3+A_4) \; .
\end{eqnarray}
Then, \eqref{besov1} and \eqref{deriv2} together lead to 
\begin{eqnarray*}
|A_1|+|A_2| &\lesssim & \|\partial_x([\H\partial_x  ,\Psi] u)\|_{L^2} \ \|u^2\|_{L^2_x} + \|\Psi_{2x}\|_{L^\infty_x} \ \|u\|_{L^2_x} \ \|u^2\|_{L^2_x}\\ 
& \lesssim &  (\|D_x^{2-\varepsilon}\Psi\|_{L^\infty_x}+ \|D_x^{2+\varepsilon}\Psi \|_{L^\infty_x}+  \|\Psi_{2x}\|_{L^\infty_x} ) \|u\|_{L^2}^2 \ \|u\|_{H^1} \\
&\lesssim &  \frac{\|u\|_{H^1}^3}{(R+\frac{1}{8}(t_0-t))^{(2-\varepsilon)\theta}} \ ,
\end{eqnarray*}
while \eqref{comut} together with Young's inequality leads to 
\begin{eqnarray*}
|A_3| &\lesssim & \|\partial_x(\sqrt{\Psi_x})\|_{L^\infty_x} \  \|(\sqrt{\Psi_x} u_x\|_{L^2_x} \ \|u^2\|_{L^2_x}
\\ 
& \le  & 2^{-8} \int_{\R} \Psi_x u_x^2 +  \frac{c \ \|u\|_{H^1}^4}{(R+\frac{1}{8}(t_0-t))^{3\theta}}\; .
\end{eqnarray*}
Finally, by Young's inequality, \eqref{loc} and \eqref{support}, for $ R\ge R_0 $ it holds
\begin{eqnarray*}
|A_4|& \le&  \frac{1}{2} \|u\|_{L^\infty (\text{supp}\,  \Psi_x(t))} \Bigl(\int_{\R} \Psi_x u^2+\int_{\R} \Psi_x u_x^2 \Bigr) \\
& \le & 2^{-7} \Bigl(\int_{\R} \Psi_x u^2+\int_{\R} \Psi_x u_x^2 \Bigr) \ .
\end{eqnarray*}
Gathering the above estimates, for  $ R\ge R_0 $, we get 
\begin{eqnarray*}
\Bigl| \frac{\gamma}{2} \int \H(\Psi u)_x u u_x &+ &\frac{\gamma}{2} \int \H (\Psi u u_x) u_x\Bigr|  \leq  2^{-8} \Bigl(\int_{\R} \Psi_x u^2+\int_{\R} \Psi_x u_x^2 \Bigr)+  \frac{c \|u\|_{H^1}^3(1+\|u\|_{H^1})}{(R+\frac{1}{8}(t_0-t))^{(2-\varepsilon)\theta}} \ .
\end{eqnarray*}
On the other hand, thanks to \eqref{loc} and  \eqref{Psit}, the remaining part of the right-side  of the first line in \eqref{finH1} goes for $ R\ge R_0 $  like
\begin{eqnarray}
\displaystyle \frac{1}{2} \int \Psi u_x^3 \ - \int \Psi u_x (u u_x)_x &=& - \frac{1}{2} \int \Psi u_x^3 \ - \int \Psi u u_x u_{2x} = -\frac{1}{2} \int \Psi u_x^3 \ -\frac{1}{2} \int \Psi u \partial_x(u_x^2)  \nonumber \\
\displaystyle &= &  -\frac{1}{2} \int \Psi u_x^3 \ + \frac{1}{2} \int \Psi_x u u_x^2 \ + \frac{1}{2} \int \Psi u_x^3 \nonumber \\
\displaystyle & \leq & \frac{1}{2} \|u\|_{L^\infty(\text{supp} \, \Psi_x(t))} \int \Psi_x u_x^2 \nonumber \\
\displaystyle &\leq & 2^{-7} \int \Psi_x u_x^2 \ .\nonumber
\end{eqnarray}
Recalling \eqref{finH1} and using that $ |\gamma|\le 1/2$, we finally end up for $ R\ge R_0 $ with
\begin{eqnarray}
\frac{d}{dt} J_{t_0}^{+R}(t)&\leq & 2^{-5} \int \Psi_x u^2 +2 \int \Psi_x u_x^2 -   \int \Psi_x u_{2x}^2 \nonumber \\
\displaystyle &+& \frac{c\|u\|_{H^1}^2
(1+\|u\|_{H^1}^2)}{(R+ (t_0 -t))^{(2-\varepsilon) \theta}} \ + \frac{c \ \|u\|^2_{H^1}}{(R+\frac{1}{8}(t_0-t))^{2-\theta}} \; .
\label{final H11}
\end{eqnarray}
Combining this last estimate with \eqref{finalder}, we thus get 
\begin{eqnarray}
\frac{d}{dt} (4 I_{t_0}^{+R}(t)+J_{t_0}^{+R}(t))&\leq & -2^{-4} \int \Psi_x u^2 -2\int \Psi_x u_x^2 -   \int \Psi_x u_{2x}^2 \nonumber \\
\displaystyle &+& \frac{c\|u\|_{H^1}^2
(1+\|u\|_{H^1}^2)}{(R+ (t_0 -t))^{(2-\varepsilon) \theta}} \ + \frac{c \ \|u\|^2_{H^1}}{(R+\frac{1}{8}(t_0-t))^{2-\theta}} \; 
\label{final H1}
\end{eqnarray}
which consequently yields \eqref{monotH1} by taking $\theta=3/4 $ and integrating between $ t $ and $t_0 $. Finally, \eqref{monotH12} follows by exactly the same calculations, replacing $ \Psi $ by $\tilde{\Psi} $ defined in \eqref{defPsi2} and recalling that  \eqref{dodo} together with \eqref{loc}  ensure that  $|u(t,x)| < 2^{-6} $ for $ x\in \text{supp}\,  \tilde{\Psi}_x(t) $  with $ t\ge t_0$.
\end{proof}
\section{The Nonlinear Liouville theorem}

\subsection{Some properties of the branches of solitary waves for $ |\gamma| $ small}\label{subsectionsmooth}

\begin{proposition}\label{der}
There exist $ \gamma_0>0, \alpha_0>0 $, $M_0$ and $ \delta_0>0 $  such that 
\begin{equation}\label{smooth}
(\gamma,c) \mapsto Q_{\gamma,c}  \text{ is of class $C^1$ from }
 ]-\gamma_0,\gamma_0[ \times ]1-\delta_0, 1+\delta_0[ \text{ into } 
H^2(\R) \; .
\end{equation}
with $  \forall (\gamma,c_*)\in  ]-\gamma_0,\gamma_0[ \times ]1-\delta_0, 1+\delta_0[$, 
\begin{equation}\label{tr2}
\frac{d}{dc}_{|c=c_*} \| Q_{\gamma,c}\|_{L^2}^2>\alpha_0  
\end{equation}
and 
\begin{equation}\label{tr3}
 \sup_{x\in\R}  \Bigl(x^2 |Q_{\gamma,c}(x)|+ |x^3| (|Q_{\gamma,c}'(x)|+|Q_{\gamma,c}''(x)| )\Bigr) \le M_0\; .
\end{equation}
Moreover, for any $   (\gamma,c)\in ]-\gamma_0,\gamma_0[ \times ]1-\delta_0, 1+\delta_0[$, $Q_{\gamma,c} $ is orbitally stable.
\end{proposition}
\begin{proof} The proof of the existence, orbital stability and the decay of these solitary waves can be found for instance in \cite{ bonachen} or \cite{ADM1} (see also \cite{ABR}, \cite{Angulo}) . Even if the uniformity of the constant $M_0 $ appearing in \eqref{tr3} for $ (\gamma,c) $ in the above set is not clearly established in these papers, following step by step the proof of (\cite{ADM1}, Proposition 4.1) one can easily check that it holds, for instance,  as soon as we take $ \delta_0\le 1/4 $ and $ \gamma_0\le 1 $. 

 Now, we need to prove  the $C^1$-regularity of $(\gamma, c)\mapsto Q_{\gamma,c} $ that seems not reachable by the approach followed in \cite{ADM1}. 
 To this aim we give another proof of the existence of solitary wave profiles based on an  application of the Implicit Function Theorem as in \cite{ABR} (see also \cite{Maris} or \cite{Kabakouala Molinet 2018} for a similar application). 
 
For $ s\ge 0  $, we denote by $H^s_e(\R)=\left\{u\in H^{s}(\mathbb{R}):u(-\cdot)=u(\cdot)\right\}$ the closed subspace of even functions of $ H^s(\R) $. 
For $\gamma>0$, we define the map 
$$
T \quad :\quad 
\begin{array}{rcl}
\R^2 \times H^2_e(\R)&\rightarrow&  H^0_e(\R) \\
 (\gamma, c,\psi)&\mapsto  & c\psi - \partial^{2}_{x}\psi-\gamma \H \partial_{x}\psi-\frac{1}{2}\psi^{2}
 \end{array} \; .
 $$
We then have   $T(0,1,Q_{KdV})=0$, where $ Q_{KdV} $ is the  profile of the KdV soliton of speed $1$. 
 Moreover,  $T$ is  of class $C^{1}$ and it holds
\begin{equation}\label{exi1}
\partial_{\psi}T(0,1,Q_{KdV})=\mathcal{L}\vert_{H^2_e(\R)} ,
\end{equation}
%where $\mathcal{L}_{\vert_{H^2_e(\R)}}$ denotes the restriction of 
%$$
%\mathcal{L} \; : \; H^2(\R) \to L^2(\R) , %\; 
%\psi \mapsto  \psi - \partial^{2}_{x}\psi-Q_{KdV} \psi 
%
where $\mathcal{L}_{\vert_{H^2_e(\R)}}$ denotes the restriction on $H^2_e(\R)$ of
%$$ \mathcal{L} \; : \; H^2(\R) \to L^2(\R) , \ \psi \mapsto  \psi - \partial^{2}_{x}\psi-Q_{KdV} \psi $$
$$
\begin{array}{rcl}
\mathcal{L} \; :  H^2(\R) & \rightarrow & L^2(\R) \ , \\
 \psi & \mapsto  & \psi - \partial^{2}_{x}\psi-Q_{KdV} \psi 
 \end{array} \; .
 $$
Clearly, the spectrum of $ \mathcal{L}\vert_{H^2_e(\R)}$  is included in the spectrum of  $\mathcal{L}$.
Now, it is well-known (see for instance \cite{Bona et al 1987} or \cite{Weinstein87}) that the essential spectrum of $\mathcal{L}$ is $ [1,+\infty[ $ and that $ \ker \mathcal{L}={\rm span}\,   (Q'_{KdV}) $. Since $ Q'_{KdV} $ is an odd function, it follows that $ \mathcal{L}\vert_{H^2_e(\R)}$ is an isomorphism from 
$H^2_e(\R) $ into $ H^0_e(\R) $. Therefore, the Implicit Function Theorem ensures that there exist  $ \gamma_0, \ \delta_0, \ \mu_0 > 0 $,  and a $ C^1$-map 
$$ R \, :\, ]-\gamma_0, \gamma_0[\times ]1-\delta_0, 1+\delta_0[ \ \rightarrow H^2_e(\R) $$
 such that for all $(\gamma,c,\psi) \in  
 ]-\gamma_0, \gamma_0[\times ]1-\delta_0, 1+\delta_0[ \times B(Q_{KdV},\mu_0)_{H^2_e(\R)} $,
 $$
 T(\gamma,c,\psi)=0 \Leftrightarrow \psi=R(\gamma,c) \; .
 $$
 For $ (\gamma,c)\in ]-\gamma_0, \gamma_0[\times ]1-\delta_0, 1+\delta_0[ $, we set $ \tilde{Q}_{\gamma,c} = R(\gamma,c) $. By construction,  $ \tilde{Q}_{\gamma,c}\in H^2_e(\R) $ is an even solution to \eqref{eqQ}.
 Now, according to \cite{ADM1}, there exists $ \tilde{\gamma}_0 >0 $ such that for $ |\gamma|< \tilde{\gamma}_0 $ there exists a unique even ground state solution $ Q_{\gamma,1} $  to \eqref{eqQ} with $c=1 $. Moreover,  $ Q_{\gamma,1} $ is orbitally stable
  and $  Q_{\gamma,1}\tendsto{\gamma\to 0} Q_{KdV} $ in $H^2(\R) $\footnote{Only the $ H^1(\R) $ convergence is actually  proven in \cite{ADM1}. However, since the boundedness of $ \{Q_{\gamma,1},  |\gamma|< \tilde{\gamma}_0\} $  in $H^1(\R) $ leads directly to the boundedness of this same set in $ H^3(\R) $, the convergence in $ H^2(\R) $ follows easily } . Since $ \psi $ is a solution to \eqref{eqQ} with $ c=1 $ and $ \gamma=\beta $ if and only if $ x\mapsto \lambda^2 \psi (\lambda\cdot ) $ is a solution to  \eqref{eqQ} with $ c=\lambda^2 $ and $ \gamma= |\lambda| \beta$, we deduce that for $ \epsilon>0 $ small enough there exist $ \gamma_{\epsilon}, \delta_{\epsilon}>0 $ such that for all $ |\gamma|< \gamma_\epsilon $ and all
  $c\in ]1-\delta_0, 1+\delta_0[ $, \eqref{eqQ} has a unique even ground state solution $ Q_{\gamma,c} $ that is orbitally stable and satisfies 
 $ \|Q_{\gamma,c} -Q_{KdV} \|_{H^2_e} < \epsilon $. Therefore, by uniqueness there exists $  \tilde{\gamma}_0, \ \tilde{\delta}_0>0 $ such that for 
   all $ |\gamma|<  \tilde{\gamma}_0 $ and all
  $c\in ]1- \tilde{\delta}_0, 1+ \tilde{\delta}_0[ $, $\tilde{Q}_{\gamma,c}=Q_{\gamma,c} $ and thus $ (c,\gamma)\mapsto Q_{\gamma,c} $ is of class $ C^1$ on this open set with values in $ H^2(\R) $.
 We thus infer from classical results on the KdV equation (see for instance \cite{Bona et al 1987} or  \cite{Weinstein87}) and  the $ C^1$-smoothness of $ (\gamma,c) \mapsto Q_{\gamma,c}$  that  there exist $  \check{\gamma}_0, \check{\delta}_0>0 $ such that for 
   all $ |\gamma|<  \check{\gamma}_0 $ and all
  $c_0\in ]1-\check{\delta}_0, 1+\check{\delta}_0[ $, 
  $$
\frac{d}{dc}_{\vert c=c_0}\|Q_{\gamma,c}\|_{L^2}^2= 2 (\partial_{c_{\vert c=c_0}} Q_{\gamma,c}, Q_{\gamma,c_0}) \ge   (\partial_{c_{\vert c=1}}Q_{c,KdV}, Q_{c,KdV}) >0 \; ,
  $$
  and 
  \begin{equation}\label{derder}
  \|Q_{\gamma,c}-Q_{\gamma,1}\|_{H^2} \le \sup_{c_*\in ]1-\check{\delta}_0, 1+\check{\delta}_0[ }
  \|\partial_{c_{\vert c=c_*}} Q_{\gamma,c}\|_{H^2} |c-1| \le 2 \|\partial_{c_{\vert c=1}} Q_{KdV,c}\|_{H^2}|c-1| \; .
 \end{equation}

 \end{proof}
\subsection{Almost monotonicity for a perturbed equation and $ H^1$-decay of its $H^1(\R) $-compact solutions}
We will need the following monotonicity result for a slight linear modification of  \eqref{MainEq}. 
\begin{proposition}\label{propdecay}
Let $ 0<b<2^{-6} $, $ x(\cdot) $ be a $ C^1$-function,  and $ a(\cdot) $ be a continuous function on $ \R $ with $ |\dot{x}(\cdot)-1|\ll 1  $ and $|a(\cdot)| \ll 1 $.
Let also $ |\gamma|\ll 1 $, $ |c-1|\ll 1$  and let $ {\tilde v}\in C(\R;H^1(\R)) $ be a solution to 
\begin{equation}\label{eqv}
\partial_t {\tilde v} + \partial^3_x{\tilde v}+\gamma \H \partial_x^2  {\tilde v} - \dot{x}(t)  \partial_x {\tilde v}+ a(t) Q_{\gamma,c}'+ \partial_x ( {\tilde v}  Q_{\gamma,c}) +  \frac{b}{2} \partial_x (  {\tilde v}^2)=0
\end{equation}
with $\displaystyle  \sup_{t\in\R} \|{\tilde v}(t)\|_{H^1} \lesssim 1 $ such that   $ \forall \delta>0 $, $ \exists R_\delta>0 $ such that 
\begin{equation}\label{dse}
\forall t\in \R , \quad \int_{|x|>R_\delta} {\tilde v}^2(t,x)+{\tilde v}_x^2(t,x) \, dx < \delta \; .
\end{equation}
Then  there exists a  constant $ C>0 $, that does not depend on $b$, $\gamma$, $c$, $x(\cdot)$, $a(\cdot) $ or the  map $ \delta\mapsto R_\delta $ such that for any $ R>0 $ it holds 
$$
\forall t\in \R , \quad \int_{|x|>R} {\tilde v}^2(t,x)+{\tilde v}_x^2(t,x) \, dx < C R^{-1/4} \; .
$$
\end{proposition}
\begin{proof} We first prove an almost monotonicity result for the associated energy.  
Since $ |\dot{x}(t)-1|\ll 1 $, by the simple change of frame  $v(t,x)=\tilde{v}(t,x-x(t)) $ we recover the same situation as in Propositions \ref{L2Monot} and \ref{H1Monot}  with three additional  terms:
$$ -( \dot{x}(t)-1)  \partial_x v \ , \   a(t) Q_{\gamma,c}' (\cdot-x(t))\ , \ \text{and} \ \partial_x ( v  Q_{\gamma,c}(\cdot-x(t))) \ ,$$
 and a coefficient $ 0<b<2^{-6} $  in front of the nonlinear term. 
Indeed, $ v $ satisfies 
\begin{equation}\label{eqvt}
\partial_t v + \partial^3_xv +\gamma \H \partial_x^2   v  -( \dot{x}(t)-1)  \partial_x v+ a(t) Q_{\gamma,c}'(\cdot-x(t))+ \partial_x ( v  Q_{\gamma,c}(\cdot-x(t))) +  \frac{b}{2} \,\partial_x (  v^2)=0 \; .
\end{equation}
It is thus  not too hard to check that our calculations in the proofs of Propositions \ref{L2Monot} and \ref{H1Monot} still hold for $v$ by replacing the local energy $J_{t_0}^R $ by the rescaled local energy 
$$\widetilde{J}_{t_0, b}^R :=  \frac{1}{2} \int \Psi \ v_x^2 \ - \frac{\gamma}{2} \int \H(\Psi \ v)_x v \  - \frac{b}{6} \int \Psi \ v^3  \; ,$$
where $ \Psi $ is defined as in \eqref{defPsi}. However, at this stage  it is very important to notice that, in sharp contrast to the proofs in  Propositions \ref{L2Monot} and \ref{H1Monot}, we do not need to make use of \eqref{loc} in these calculations  but  instead make use of the fact that $ 0<b<2^{-6} $ to get the desired estimates on some  terms involving  $ \partial_x(v^2)$ as in \eqref{est4}, \eqref{estG}-\eqref{estH2}. Hence, the estimates  hold for all $ R\ge 1$ and not only for $ R \ge R_0$.

Therefore, it remains to control the contributions of the three new terms to the evolution of $\displaystyle \tilde{I}_{t_0}^R:= \int_{\R} \Psi v^2 $ and $\widetilde{J}_{t_0, b}^R $. First, it is easy to check that for $ \sup_{t\in\R} |\dot{x}(t)-1| $ small enough, the contribution of the first term $-( \dot{x}(t)-1)  \partial_x v$  can be absorbed by the two first terms of the right-hand side member of \eqref{final H1} and is thus harmless.
In fact, it holds 
$$
\Bigl| (\dot{x}(\cdot)-1) \int_{\R} \Psi v  \partial_x v\Bigr|=\Bigl|\frac{(1-\dot{x}(\cdot) )}{2}\Bigr|\int_{\R} \Psi_x\,  v^2 
\le \displaystyle\sup_{t\in\R} |\dot{x}(t)-1| \int_{\R}  \Psi_x\,  v^2 \ ,
 $$
 and in the same way, integrating by parts and making use of \eqref{Est3}, we get 
 \begin{align*}
 \Bigl| (\dot{x}(\cdot)-1)& \int_{\R} [\Psi (-v_{xx} +\frac{\gamma}{2}  \H v_x  + \frac{b}{2} v^2 )+\frac{\gamma}{2} \H(\Psi v)_x]  v_x\Bigr|\\
 &=\Bigl|  (\dot{x}(\cdot)-1)\Bigl( \int_{\R}\Psi_x v_x^2-\frac{b}{6} \int_{\R} \Psi_x v^3-\frac{\gamma}{2} \int_{\R} \Psi_x v \H v_x\Bigr)\Bigr|\\
 &\lesssim  \displaystyle\sup_{t\in\R} |\dot{x}(t)-1|(\int_{\R}  \Psi_x v^2+ \int_{\R}  \Psi_x v_x^2) +\frac{\|v\|^2_{L^2}}{(R+\frac{1}{8}(t_0-t))^{3/2}} 
 \end{align*}
which is acceptable since $ \displaystyle \int_t^{t_0}\frac{\|v\|^2_{L^2}}{(R+\frac{1}{8}(t_0-t))^{3/2}} \lesssim R^{-1/2} $.
 It thus remains to estimate the contribution of the two other new terms. 
\vspace*{2mm}\\
{\it Contribution to the estimate on $   \tilde{I}_{t_0}(t_0) - \tilde{I}_{t_0}(t) $: }\\
According to the above discussion, it remains to estimate 
$$
-\int_{t}^{t_0} \Bigl[a(\tau) \int_{\R} Q_{\gamma,c}' (\cdot-x(\tau))\Psi(\tau) v(\tau) + \int_{\R} \partial_x (Q_{\gamma,c} (\cdot-x(\tau)) v(\tau))  \Psi(\tau) v(\tau) \Bigr] d\tau  \; .
$$
Note that according to \eqref{psit}, for $ t\le t_0$, $ \Psi(t,x)=0 $ for $x\le x(t_0) +R +\frac{5}{8}(t-t_0) $ whereas $Q_{\gamma,c} (\cdot-x(t))$ is centered at $ x(t)\le x(t_0)+\frac{5}{6}(t-t_0)$ since $|\dot{x}(\cdot)-1|\ll 1 $.
In view of the decay of $Q_{\gamma,c}' $ given by \eqref{tr3}, after applying Cauchy-Schwarz inequality in space,  the first term above can thus be bounded by 
\begin{align*}
\Bigl| \int_{t}^{t_0} a(\tau) \int_{\R} Q_{\gamma,c}'  (\cdot-x(\tau))\Psi(\tau) v(\tau) \, d\tau   \Bigr|
& \lesssim \sup_{\tau\in\R} \|v(\tau)\|_{L^2} \int_t^{t_0} \|Q_{\gamma,c}' (\cdot-x(\tau)) \Psi (\tau)\|_{L^2_x} \, d\tau\\
& \lesssim \int_t^{t_0}\Bigl(  \int_{R} |Q'_{\gamma,c}(x)| (R+\frac{1}{8} (t_0-t))^{-3} \, dx \Bigr)^{1/2} \\
 & \lesssim  \|Q_{\gamma,c}'\|_{L^1}^{1/2} \int_{t}^{t_0} \Bigl(R+\frac{1}{8}(t_0-\tau )\Bigr)^{-3/2}d\tau \lesssim R^{-1/2}\; .
\end{align*}
For the second term, we first notice that by integration by parts
\begin{align}\nonumber
\displaystyle  \int_{\R}   \partial_x(Q_{\gamma,c}  (\cdot-x(\tau))v) \ ( \Psi v) \ dx & =   \int_{\R}  Q_{\gamma,c}'  (\cdot-x(\tau))\ v^2 \Psi \ dx +  \int_{\R}  Q_{\gamma,c}  (\cdot-x(\tau))\Psi  v v_x \ dx \nonumber \\
\displaystyle & =  \frac{1}{2}  \int_{\R}  Q'_{\gamma,c}  (\cdot-x(\tau))\ v^2  \Psi \ dx  - \frac{1}{2}  \int_{\R}  Q_{\gamma,c}  (\cdot-x(\tau))v^2  \Psi_x \ dx \ .\nonumber
\end{align}
Besides, the decay of $Q_{\gamma,c} $ and $Q_{\gamma,c}' $ given by \eqref{tr3} lead to
$$
\Bigl| \displaystyle \int_{t}^{t_0}  \int_{\R}  Q'_{\gamma,c}  (\cdot-x(\tau))\ v^2 \ \Psi \ dx \ d\tau  \Bigr|\lesssim  \int_{t}^{t_0}   \Bigl(R+\frac{1}{8}(t_0-\tau )\Bigr)^{-3} \ \|v\|^2_{L^2} \ dt' \lesssim R^{-2} 
$$
and
$$
\displaystyle \bigg | \int_{t}^{t_0} \int Q_{\gamma,c} (\cdot-x(\tau)) \ v^2 \ \Psi_x \ dx \ dt \bigg | \lesssim  \int_{t}^{t_0}   \Bigl(R+\frac{1}{8}(t_0-\tau )\Bigr)^{-2} \ \|v\|^2_{L^2} \ dt' \lesssim R^{-1} .
$$
This proves that 
$$
  \int_{\R} (\Psi v^2)(t_0) - \int_{\R} (\Psi v^2)(t) \lesssim  R^{-1/4} 
$$
which in turn leads to 
\begin{equation}\label{decL2}
\sup_{t\in\R} \int_{|x|>R} v^2(t,x) dx\lesssim R^{-1/4} 
\end{equation}
by letting $ t\to -\infty $, using \eqref{dse} together with \eqref{Psit22},  and the fact that $ \dot{x}\ge 5/6 $. \vspace*{2mm}\\
{\it Contribution to the estimate of $\widetilde{J}_{b}(t_0)-\widetilde{J}_{b}(t)$:}\\
 To lighten  the notations, we write $ Q_{\gamma,c} $ for 
$  Q_{\gamma,c}(\cdot-x(\tau)) $ in the following calculations. 
Here, it remains to control 
\begin{align*}
K_1(t)& = \Bigl| \int_{\R} \Psi v_x Q_{\gamma,c}''  + b \int_{\R} \Psi v^2 Q_{\gamma,c}'  -\frac{\gamma}{2} \Bigl( \int_{\R} \Psi  Q_{\gamma,c}' \H v_x 
+ \int_{\R} \H ( \Psi u)_x Q_{\gamma,c}' \Bigr) \Bigr| \\
&\le \Bigl| \int_{\R}  \Psi \ v_x   Q_{\gamma,c}'' \Bigr|+ \Bigl|\int_{\R}  \Psi v^2  Q_{\gamma,c}' \Bigr| +\Bigl| \int_{\R} \Psi Q_{\gamma,c}' \H v_x   \Bigr|+\Bigl| \int_{\R}  (\Psi v)_x \H Q_{\gamma,c}'\Bigr| 
\end{align*}
and 
$$
K_2(t)=\Bigl| \int_{\R}  \Psi \ v_x   \partial_x^2 ( v  Q_{\gamma,c}) \Bigr|+ \Bigl|\int_{\R}
 \Psi v^2  \partial_x ( v  Q_{\gamma,c}) \Bigr| +\Bigl| \int_{\R}  \Psi \H v_x \,  \partial_x ( v  Q_{\gamma,c})    \Bigr|+\Bigl| \int_{\R}  \H (\Psi v)_x  \partial_x ( v  Q_{\gamma,c})  \Bigr| \ . $$
To bound $ K_1 $, we notice  that since $ Q_{\gamma,c} $ satisfies \eqref{eqQ} then $ \H  Q_{\gamma,c}' $ has at least the same decay rate as $  Q_{\gamma,c} $. Therefore, the same considerations as above lead to 
$$
\int_{t}^{t_0} K_1(\tau)\, dt \lesssim  \int_{t}^{t_0}   \Bigl(R+\frac{1}{8}(t_0-\tau )\Bigr)^{-2} \ \|v\|_{H^1} \ d\tau \lesssim R^{-1} .
$$
We now tackle $K_2 $. The three first terms in $ K_2$ can be treated as above by using the decay of $Q_{\gamma,c}$ and its derivatives.
The treatment of the fourth term is more tricky. We proceed as in the proof of Proposition \ref{H1Monot} to get 
\begin{align*}
\Bigl|\int_{\R}  \H (\Psi v)_x  \partial_x ( v  Q_{\gamma,c}) \Bigr| &= \Bigl|
\int_{\R}  \H (\Psi_x v)  \partial_x ( v  Q_{\gamma,c})+
\int_{\R}  \H(\Psi v_x) \partial_x ( v  Q_{\gamma,c})\Bigr|\\
& =\Bigl|  \int_{\R} [\H,\Psi_x] v \,  \partial_x ( v  Q_{\gamma,c})+
\int_{\R}  \H v \, \Psi_x  \partial_x ( v  Q_{\gamma,c}) \\
& \hspace*{10mm}+\int_{\R} [\H,\Psi] v_x \partial_x ( v  Q_{\gamma,c})+\int_{\R} \Psi \H v_x \,  \partial_x ( v  Q_{\gamma,c})\Bigr|\\
& \le \Bigl| \int_{\R} [\H,\Psi_x] v \,  \partial_x ( v  Q_{\gamma,c})\Bigr| +
 \Bigl| \int_{\R}  \H v \, \Psi_x  \partial_x ( v  Q_{\gamma,c})\Bigr| \\
& \hspace*{10mm}+ \Bigl| \int_{\R} \partial_x([\H,\Psi] v_x)   v  Q_{\gamma,c}\Bigr|+ \Bigl| \int_{\R} \Psi \H v_x \,  \partial_x ( v  Q_{\gamma,c})
 \; \end{align*}
 so that  the same arguments as above together with the commutator estimates \eqref{comut},\eqref{besov1} and \eqref{deriv2} lead to 
 $$
 \Bigl|\int_{\R}  \H (\Psi v)_x  \partial_x ( v  Q_{\gamma,c}) \Bigr|\lesssim   \Bigl(R+\frac{1}{8}(t_0-\tau )\Bigr)^{-
 \frac{3}{4}(2-)} \; .
 $$
Eventually, we end up with 
$$
\int_{t}^{t_0} K_2(\tau) \, d\tau \lesssim R^{-\frac{1}{4}} \;.
$$
In view of Propositions \ref{L2Monot} and \ref{H1Monot}, this ensures that for $ t\le t_0 $ and $ \theta\ge 4 $, 
\begin{equation}\label{tftf}
 \theta  \ \tilde{I}^{+R}_{t_0}(t_0)+ \tilde{J}^{+R}_{t_0,b}(t_0) \leq  \theta \ \tilde{I}^{+R}_{t_0}(t) + \tilde{J}^{+R}_{t_0,b}(t)  +C_\theta \ R^{-\frac{1}{4}} \;  .
\end{equation}
Letting $ t $ tends to $-\infty $ and taking  \eqref{dse} into account, we acquire
\begin{equation}\label{nono}
 \theta  \tilde{I}^{+R}_{t_0}(t_0)+ \tilde{J}^{+R}_{t_0,b}(t_0) \leq  C_\theta  \ R^{-\frac{1}{4}} \ ,  
\end{equation}
  where we used  that  $ \displaystyle |\int_{\R} \Psi v \H v_x | \le \|v\|_{H^1} \|v\|_{L^2(\text{supp}\, \Psi)} \tendsto{t \to -\infty} 0 $ thanks to \eqref{Psit}, \eqref{loc} and the fact that $ \dot{x} \ge 5/6 $. Now, noticing that \eqref{comut} together with \eqref{deriv2} leads to 
\begin{align}
|\int_{\R} \Psi v \H v_x| & \le \|\sqrt{\Psi} v\|_{L^2} ( +\|\sqrt{\Psi} v_x\|_{L^2} )\nonumber \\
& \le 2 \|\sqrt{\Psi} v\|_{L^2}^2+\frac{1}{4} \|\sqrt{\Psi} v_x\|_{L^2}^2 +\Bigl\|\frac{\Psi'}{ \sqrt{\Psi}}\Bigr\|_{L^\infty} \|v\|_{L^2}^2 \nonumber  \\
& \le 2 \|\sqrt{\Psi} v\|_{L^2}^2+\frac{1}{4} \|\sqrt{\Psi} v_x\|_{L^2}^2 +C \, R^{-3/4} \; , \label{estgg}
\end{align} 
and that, since by  hypothesis  and Sobolev embedding $ \|v \|_{L^\infty}\le  \|v  \|_{H^1} =  \|\tilde{v} \|_{H^1} \lesssim 1 $,   it holds that $|v ^3|\le  \|v \|_{H^1} v^2 $, we infer that taking  $ \theta_0\ge 4+\|v\|_{H^1}  $,
\begin{equation}\label{nono1}
 \theta_0 \tilde{I}^{+R}_{t_0}(t_0)+ \tilde{J}^{+R}_{t_0,b}(t_0) \ge \frac{1}{4} \int_{\R} \Psi (t_0) (v^2+v_x^2)(t_0)- C\,   R^{-1/4} \; .
\end{equation}
Combined with \eqref{nono}, this ensures that 
$$
\int_{\R} \Psi (t_0) (v^2+v_x^2)(t_0) \lesssim R^{-1/4} \; 
$$
which in turn shows the decay of the local $ H^1$-norm at the right in view of \eqref{Psit22}.
Finally, the decay in the left of the $H^1$-compact solution $ \tilde{v}$ follows from the invariance of \eqref{MainEq} under the transformation $ (t,x) \mapsto (-t,-x) $.
\end{proof}
\subsection{Modulation around the solitary wave}
\begin{lemma}\label{parameter}
Let $u\in C(\R;H^1(\R)) $ be a solution of \eqref{MainEq}. There exist $C, \delta_0 >0$ such that for any $|\gamma| \ll 1 $, $ |c-1|\ll1 $ and $0<\delta <\delta_0$, if $u $ satisfies
\begin{equation}
\inf_{r \in \mathbb{R}} \|u(t) - Q_{\gamma,c} (.-r)\|_{H^1} \leq \delta,  \quad \forall \ t \in \mathbb{R} ,
\end{equation}
then there exists a $ C^1$-function $\rho: \mathbb{R} \ni t \mapsto\rho(t) \in\R $ such that
\begin{equation} \label{decomp}
\eta(t,\cdot) := u(t,\cdot+ \rho(t)) - Q_{\gamma,c}(\cdot)
\end{equation}
enjoys the following for all $t \in \mathbb{R}$ :
\begin{equation} \label{decomp2}
 \|\eta(t)\|_{H^1} \leq C \ \delta\; , 
\end{equation}
\begin{equation} \label{decomp3}
\displaystyle \int Q_{\gamma,c}'(x) \ \eta(t,x) \ dx=0 , 
\end{equation}
\begin{equation} \label{decomp4}
 \text{and} \quad \quad |\dot{\rho}(t) -c| \leq C \ \|\eta(t)\|_{L^2}\; . 
\end{equation}

\end{lemma}
\begin{proof}
The proof follows from standard arguments (see for instance \cite{Weinstein87} or  \cite{Bona et al 1987}  for the regularity of $ \rho(\cdot)$). For $r >0$ and $ v\in H^1(\R)$, we denote by $B(v,r)_{H^1}$ the open ball of $ H^1(\R)$ centered at $ v $,  that is 
$$ B(v,r)_{H^1} := \{w \in H^1(\R); \ \|w - v\|_{H^1} \leq r \} \; .$$
Let us fix  $z \in \mathbb{R}$ and  define  the application $ F_z $ from $\R \times H^1(\R) $ into $ \R $ by 
$$F_z(s,u):= \int Q_{\gamma,c}'(x-z) (u(x+s)- Q_{\gamma,c}(x-z)) \ dx .$$
We notice that $F_z(0,Q_{\gamma,c} (\cdot-z))=0$ and that for $ |c-1|\ll 1 $ and $|\gamma|\ll 1$, 
\begin{equation} \label{low}
\displaystyle \frac{\partial}{\partial_s} F_z(s,u)|_{s=0,u=Q_{\gamma,c}(\cdot-z)} = \int (Q^{'}_{\gamma,c})^2\ge \frac{1}{2} 
 \int (Q^{'}_{KdV})^2>0,
\end{equation}
where we made use of \eqref{smooth} to have a lower bound that does not depend  on  $ \gamma$ nor on $ c$.
 Hence, by the Implicit Function Theorem, there exists $ \delta_z>0 $ and a unique  $C^1$-mapping , 
 $$ s_z: B(Q_{\gamma,c}(\cdot-z),\delta_z)_{H^1} \longrightarrow \mathbb{R} $$
  such that $F_z(s_z(u),u) =0$ on $ B(Q_{\gamma,c}(\cdot-z),\delta_z)_{H^1} $. Furthermore, there exits $ C_z>0 $ such that  
 \begin{equation}\label{tv}
 |s_z(u)| \leq C_z \ \|u-Q_{\gamma,c}(\cdot-z)\|_{H^1} \; .
 \end{equation}
 Note that by translation invariance, $ \delta_z $ and $C_z $ do not depend on $ z\in\R $, that is $ \delta_z=\delta$ and $C_z=C$. Therefore by uniqueness,  we can define a $ C^1$-mapping from 
$$ V_{\delta} := \{u \in H^1(\R); \ \inf_{r \in \mathbb{R}} \|u - Q_{\gamma,c}(.-r)\|_{H^1} <\delta \} $$
into $ \R $ by setting 
$s(u)= z+s_z(u) $ for any $ u\in B(Q_{\gamma,c}(\cdot-z),\delta)_{H^1} $. 
 Now, to prove that $ \rho= s\circ u $ is of class $ C^1 $, we have to repeat this argument but by considering $ \tilde{F}_z\, :\, \R \times H^{-2}(\R) \to \R $ defined by 
 $$\tilde{F}_z(s,u):=  \langle Q_{\gamma,c}'(\cdot-z) \, ,\,  (u(x+s)- Q_{\gamma,c}(\cdot-z)) \rangle_{H^2,H^{-2}} \;.$$
 Since $Q_{\gamma,c}\in H^\infty(\R) $, $\tilde{F}_z $ is well defined and of class $ C^1 $, with   $\tilde F_z\equiv F_z $ on $\R\times H^1(\R) $. By the Implicit Function Theorem, we thus deduce that there exists $ \tilde{\delta}>0 $ and a $  C^1$-function $ \tilde{s} $ from 
 $$ \tilde{V}_{\tilde{\delta}} := \{u \in H^{-2}(\R) ; \ \inf_{r \in \mathbb{R}} \|u - Q_{\gamma,c}(.-r)\|_{H^{-2}} <\tilde{\delta} \} $$
 into $\R $, 
such that 
$$
 \langle Q_{\gamma,c}'(\cdot-y) \, ,\,  (u- Q_{\gamma,c}(\cdot-y)) \rangle_{H^2,H^{-2}} \Leftrightarrow y=\tilde{s}(u) , \quad \forall u\in \tilde{V}_{\tilde{\delta}} \; .
$$
 We set $ \tilde{\tilde{\delta}}= \delta\wedge \tilde{\delta} $. By uniqueness, it holds $ s\equiv \tilde{s} $ on $\displaystyle \cup_{z\in\R} B(Q_{\gamma,c}(\cdot-z),\tilde{\tilde{\delta}}) $ and thus $ s $ is also a $ C^1 $-function on $\displaystyle \cup_{z\in\R} B(Q_{\gamma,c}(\cdot-z),\tilde{\tilde{\delta}}) $ equipped with the metric inducted by $ H^{-2}(\R) $. Since by the equation \eqref{MainEq}, $ u\in C(\R;H^1(\R)) $ forces 
 $ u\in C^1(\R;H^{-2}(\R)) $, it follows that $ t\mapsto \rho(t)= s(u(t)) $ is a   $ C^1$-function on $ \R $. Setting 
$$\eta(t,):= u(t,\cdot+\rho(t)) - Q_{\gamma,c} \ ,$$ 
we get by construction 
$$ \int Q_{\gamma,c}'(x) \ \eta(t,x) \ dx=0 \ , \quad \forall t\in\R \ ,$$
and \eqref{tv} together with \eqref{smooth} ensures that there exists $C>0$ such that
$$ \|\eta(t)\|_{H^1} \leq C \inf_{z\in\R}  \|u(t)-Q_{\gamma,c}(\cdot-z)\|_{H^1} \ \leq C \ \delta .$$
Finally, setting $R(t)=Q_{\gamma,c}(\cdot-\rho(t)) $ and then  $w(t)=u(t)-R(t)= \eta(\cdot-\rho(t))$, and differentiating \eqref{decomp3} with respect to time, we get 
\begin{eqnarray}
\int_{\R} w_t Q_{\gamma,c}'(\cdot-\rho(t)) &=&  \dot{\rho}(t)  \int_{\R}  w Q_{\gamma,c}''(\cdot-\rho(t))\nonumber \\
&=&  -\dot{\rho}(t)  \int_{\R}  w_x Q_{\gamma,c}'(\cdot-\rho(t))\nonumber \\
&=& (c-\dot{\rho}(t)) O(\|w\|_{H^1}) + c \, O(\|w\|_{H^1}) \; . \label{kj}
\end{eqnarray}
Substituting $ u $ by $w+R $ in \eqref{MainEq} and using that $R $ satisfies 
$$
\partial_t R +(\dot{\rho}-c) \partial_x R +\partial_x R^3  + \gamma \ \H \partial_x^2 R  + R  \partial_x R =0 \, ,
$$
we infer that $ w$ satisfies 
\begin{align*}
w_t-(\dot{\rho}-c) \partial_x R =-\partial_x^3 w -\gamma  \H \partial_x^2 w-\frac{1}{2} \partial_x \Bigl( (w+R)^2 -R^2\Bigr) \;.
\end{align*}
Taking the $ L^2$-scalar product of this last equality with $ \partial_x R $, integrating by parts and using  \eqref{kj}, we eventually get 
$$
|\dot{\rho}-c| \Bigl( \| Q_{\gamma,c}'\|^2_{L^2} +O(\|\eta\|_{H^1}) \Bigr) = O(\|\eta\|_{H^1})
$$
which proves \eqref{decomp4} thanks to \eqref{low}.
\end{proof} 

\subsection{Proof of the Nonlinear Liouville theorem}
We proceed by contradiction. First, we notice that since $(t,x) \mapsto Q_{\gamma,c}(x-ct) $ is a solution to \eqref{MainEq} and \eqref{MainEq} is translation invariant, a solution $ u $ to \eqref{MainEq} is of the form $Q_{\gamma,c}(x-z-ct)$ for some $ y\in\R $ if and only if  there exist $ t\in \R $ and $ y\in \R $ such that $u(t,\cdot) =Q_{\gamma,c}(\cdot-y) $. 

Therefore, if the statement of Theorem \ref{NL} does not hold, there exist sequences $(\gamma_n)$ with $|\gamma_n| \searrow 0$, $(\varepsilon_n) \searrow 0$, $ (u_n) \in C(\mathbb{R}; H^1(\mathbb{R})) \cap L^{\infty}(\mathbb{R}, H^1(\mathbb{R}))$ solutions of \eqref{eqbeta} with $\gamma=\gamma_n $ and $(\tilde{x}_n) $ with $ \tilde{x}_n \,:\, \R\to\R $ such that 
\begin{enumerate}
\item $ \forall n\in \N $,
\begin{equation}\label{hypo1}
\|u_n-Q_{\gamma_n,1}(\cdot -\tilde{x}_n(t))\|_{H^1} \le  \varepsilon_n \; ,
\end{equation}
\item $ \forall n\in\N, \; \forall \delta>0 $, $ \exists R_{n,\delta}>0 $ such that 
\begin{equation}\label{hypo2}
\forall t\in \R , \quad \int_{|x-\tilde{x}_n(t)|>R_{n,\delta}} u_n^2(t,x)+u_{n,x}^2(t,x) \, dx < \delta \; ,
\end{equation}
\item For any $ n\in\N $,  any $(t,y)\in\R^2 $ and any $ c>0$ with $ |c-1 |\ll 1 $,
\begin{equation} \label{comment}
u_n(t,\cdot) \neq Q_{\gamma_n,c}(\cdot-y) .
\end{equation}
\end{enumerate}
We separate the proof into different steps for the sake of clarity:\\
{\bf Step 1.} First renormalization and  Modulation. \\
We start by making  use of Proposition \ref{der} to find a solitary wave profile with the same $ L^2$-norm as $ u_{0,n} $ just by slightly modulating the speed.
Indeed, according to Proposition \ref{der},  $ \frac{d}{dc}_{|c=c_0} \|Q_{\gamma,c} \|_{L^2}^2 \ge
\alpha_0>0$ for $ 0\le \gamma\le \gamma_0 $ and $c_0\in ]1-\delta_0,1+\delta_0[ $, we hence deduce that for any $  0\le \gamma\le \gamma_0 $, $c \mapsto  \|Q_{\gamma,c} \|_{L^2}^2 $ is an increasing bijection from $  ]1-\delta_0,1+\delta_0[ $ 
  into 
  $$ ]\|Q_{\gamma,1-\delta_0}\|^2_{L^2}, \|Q_{\gamma,1+\delta_0}\|^2_{L^2}[ \supset 
 ] \|Q_{\gamma,1}\|^2_{L^2}-\alpha_0 \delta_0, \|Q_{\gamma,1}\|^2_{L^2}+\alpha_0 \delta_0[ \; .
 $$ Therefore, for  any $n\in\N $ large enough,  there exists $c_n\in\R $ with 
 \begin{equation}\label{estc}
  |c_n-1|\lesssim \varepsilon_n 
  \end{equation}
  such that 
\begin{equation}\label{renor1}
\|Q_{\gamma_n,c_n}\|_{L^2} = \|u_{0,n}\|_{L^2}  \; .
\end{equation}
It is worth noticing that \eqref{derder} then ensures that 
$$
\|Q_{\gamma_n,c_n}-Q_{\gamma_n,1}\|_{L^2} \lesssim \varepsilon_n \ ,
$$
and thus by the triangular inequality, we have 
\begin{equation}\label{renor2}
\sup_{t\in\R} \|u_n(t)-Q_{\gamma_n,c_n}(\cdot-\tilde{x}_n(t))\|_{L^2}  \lesssim \varepsilon_n
  \; .
\end{equation}
Forthwith, we apply Lemma \ref{parameter} to $ u_n $ around $Q_{\gamma_n,c_n}$. There exists $ x_n \;: \; \R\to \R $ of class $ C^1 $ such that 
\begin{equation}\label{mod1}
\sup_{t\in\R} |x_n(t)- \tilde{x}_n(t)| \lesssim \varepsilon_n \; .
\end{equation}
Moreover, setting $ \eta_n(t) = u_n(t) -Q_{\gamma_n,c_n}(\cdot-x_n(t)) $, it holds 
\begin{equation}\label{eta n}
\sup_{t\in\R} |\dot{x}_n(t)-1|\lesssim\sup_{t\in\R} \|\eta_n (t) \|_{H^1} \lesssim \varepsilon_n, \quad  
  \text{and} \quad \int_{\R} \eta_n(t,\cdot) Q'_{\gamma_n,,c_n}(\cdot-x_n(t)) =0 \; .
\end{equation}
From \eqref{hypo2}, \eqref{mod1} and the decay of $ Q_{\gamma_n,c_n}$, we infer that for any $ \delta>0 $, there exists $  R_{n,\delta}>0 $ such that 
$$
\sup_{t\in\R} \int_{|x-x_n(t)|>R_{n,\delta}} \eta_n^2(t)+ \eta_{n,x}^2(t) < \delta \; .
$$
Moreover, the  conservation of the $ L^2$-norm leads to 
$$2\int_{\R} \eta_n(t) Q_{\gamma_n,c_n} (\cdot-x_n(t))=\|u(t)\|_{L^2}^2 -\|Q_{\gamma_n,c_n}\|_{L^2}^2 -\|\eta_n\|_{L^2}^2=
\underbrace{\|u_0\|_{L^2}^2 - \|Q_{\gamma_n,c_n}\|_{L^2}^2}_{=0} -\|\eta_n\|_{L^2}^2
$$
and thus 
\begin{equation}\label{esta}
\bigg |\int_{\R} \eta_n(t) Q_{\gamma_n,c_n} (\cdot-x_n(t)) \bigg | \lesssim \|\eta_n\|_{L^2}^2 \; .
\end{equation}
Finally, since  $ u_n(t) \neq Q_{\gamma_n,c} (\cdot-y)$  for any $(t,y)\in\R^2 $ and  $ |c-1|\ll 1 $ by hypothesis, we infer  from \eqref{estc} and the definition of $ \eta_n $ that  $ \eta_n(t)\neq 0 $ for all $t\in\R $ and $ n\in\N $. 
 
We notice that the equation satisfied by $ \eta_n $ reads 
$$
\partial_t \eta_n +\partial^3_x \eta_{n} +\gamma_n \H \partial_x^2 \eta_n + (c_n-\dot{x}_n) Q_{\gamma_n,c_n}'(\cdot-x_n(t))+ (\eta_n Q_{\gamma_n,c_n}(\cdot-x_n(t))_x
 + \eta _n \eta_{n,x} =0
$$
where we made use of the fact that $ Q_{\gamma_n,c_n} $ satisfies 
$$
-c_nQ_{\gamma_n,c_n}'+Q_{\gamma_n,c_n}^{'''} +\gamma_n \H Q_{\gamma_n,c_n}^{''}+\frac{1}{2} (Q_{\gamma_n,c_n}^2)'=0 \; .
$$
Therefore $ \tilde{\eta}_n(t,\cdot)= \eta(t, \cdot+x_n(t)) $ satisfies 
$$
\partial_t \tilde{\eta}_n + \partial^3_x \tilde{\eta}_{n} +\H \partial_x^2  \tilde{\eta}_n - \dot{x}_n  \partial_x \tilde{\eta}_n+ (c_n-\dot{x}_n)Q_{\gamma_n,c_n}'+ \partial_x ( \tilde{\eta}_n  Q_{\gamma_n,c_n}) +  \tilde{\eta}_n \partial_x  \tilde{\eta}_n =0
$$
where, again  for any $ \delta>0 $, there exists $  R_{n,\delta}>0 $ such that 
 $$
\sup_{t\in\R} \int_{|x-x_n(t)|>R_{n,\delta}} \tilde{\eta}_n^2(t)+ \tilde{\eta}_{n,x}^2(t) < \delta\; .
$$
{\bf Step 2.} Second renormalization. \\
As we proved that $ \eta_n(t) \neq 0 $ for all $t\in\R $, we can consider $ t_n\in \R $ such that 
\begin{equation} \label{eta tn}
\|\eta_n(t_n) \|_{H^1}\ge \frac{1}{2} \sup_{t\in\R} \|\eta_n(t)\|_{H^1}>0 \; .
\end{equation}
Set 
$$
w_n(t)= \frac{\eta(t+t_n)}{\displaystyle \sup_{\tau\in\R}\| \eta(\tau)\|_{H^1}} \quad \text{and} \quad \tilde{w}_n(t)=w_n(t,\cdot+x_n(t))
= \frac{\tilde{\eta}_n(t)}{\displaystyle \sup_{\tau\in\R} \|\tilde{\eta}(\tau)\|_{H^1}}  \; .
$$
$\tilde{w}_n $ then satisfies 
\begin{equation}\label{eqtw}
\partial_t \tilde{w}_n + \partial^3_x \tilde{w}_{n} +\gamma_n \H \partial_x^2  \tilde{w}_n -\dot{x}_n  \partial_x \tilde{w}_n+ (c_n-\dot{x}_n) Q_{\gamma_n,c_n}'+ \partial_x ( \tilde{w}_n  Q_{\gamma_n,c_n}) +  \frac{\delta_n}{2} \partial_x (  \tilde{w}_n^2)=0
\end{equation}
with $ \delta_n= \sup_{\tau\in\R}  \| \tilde{\eta}_n(t) \|_{H^1} \tendsto{n\to \infty} 0 $.
Consequently, thanks to \eqref{eta n}, \eqref{esta} and \eqref{eta tn}, we have 
\begin{equation}\label{syst1}
\left\{
\begin{array}{l}
\|\tilde{w}_n(t)\|_{H^1} \le 1 , \quad t\in \R \vspace{2mm}\\
\| \tilde{w}_n(0)\|_{H^1} \ge 1/2 \;  \vspace{2mm}\\
\displaystyle \int_{\R} \tilde{w}_n(t) Q'_{\gamma_n,c_n} =0 , \quad \forall t\in\R  \vspace{2mm}\\
\sup_{\tau\in\R} \displaystyle |\int_{\R} \tilde{w}_n(\tau) Q_{\gamma_n,c_n} |\lesssim \sup_{\tau\in\R} \|\eta_n(\tau)\|_{H^1} \tendsto{n\to\infty} 0 
\end{array}
\right.\quad , 
\end{equation}
and  for any $ \delta>0 $, there exists $  R_{n,\delta}>0 $ such that 
 \begin{equation}\label{renorm1}
\sup_{t\in\R} \int_{|x|>R_{n,\delta}} \tilde{w}_n^2(t,x)+ \tilde{w}_{n,x}^2(t,x) \, dx < \delta\; .
\end{equation}
At this stage, we notice that since $ \delta_n \to  0 $, for $ n $ large enough $ \tilde{w}_n $ satisfies the hypotheses of  Proposition \ref{propdecay}. Therefore, there exists $ C>0 $ such that for any $n\ge 0 $ large enough  and any $ R>0 $,
\begin{equation}\label{renorm2}
\sup_{t\in\R} \int_{|x|>R} \tilde{w}_n^2(t,x)+ \tilde{w}_{n,x}^2(t,x) \, dx \le C R^{-1/4} \ .
\end{equation}
{\bf Step 3.} Passing to the limit and applying the linear Liouville theorem for the KdV equation. \\
We pass to the limit up to a subsequence on $(\tilde{w}_n)_{n\ge 0} $. Thanks to  \eqref{syst1},\eqref{renorm1} and \eqref{renorm2}, there exists $ \tilde{w} \in C(\R;L^2(\R))\cap L^\infty(\R;H^1(\R)) $ such that, up to a subsequence, 
\begin{equation}\label{contr9}
 \tilde{w}_n \longrightarrow  \tilde{w} \quad \text{in} \ C([-T,T], L^2(\R)) ,
\end{equation}
\begin{equation}\label{contr10}
 \tilde{w}_n \rightharpoonup  \tilde{w}\quad \text{in} \ C_w([-T,T], H^1(\R)).
\end{equation}
This ensures that $ \tilde{w} $ satisfies 
\begin{equation}\label{limit2}
\sup_{t\in\mathbb{R}} \int_{|x|>R} \tilde{w}^2(t,x)+ \tilde{w}_{x}^2(t,x) \, dx \le C R^{-1/4} \ .
\end{equation}
 Furthermore, since $ |\gamma_n|\to 0  $, $ c_n\to 1$  and $ |\dot{x}_n-1|\to 0 $ on $ [-T,T] $, the continuity of the map $ (\gamma,c)\mapsto Q_{\gamma,c} $ with values in $ H^2(\mathbb{R}) $ (see Subsection \ref{subsectionsmooth})  leads to 
 $$
\int_{\mathbb{R}} \tilde{w}(t) Q_{KdV}'=\int_{\mathbb{R}} \tilde{w} (t)Q_{KdV}=0 , \quad \forall t\in\mathbb{R} ,
$$
and $ \tilde{w} $ satisfies on $ \mathbb{R}^2 $, 
$$
\partial_t \tilde{w} -\tilde{w}_x +\partial^3_x \tilde{w} +\partial_x( Q_{KdV} \, \tilde{w})=0 \ .
$$
 The linear Liouville property for the KdV equation (see \cite{Martel 2006}, proof of Theorem 1) accordingly forces $ \tilde{w}=0 $.  \\
{\bf Step 4.} The final contradiction argument. \\
 Let us now show  that $ \tilde{w}= 0 $ is not compatible with  some of the properties of the sequence $( \tilde{w}_n)_{n\ge 0} $ summarized in \eqref{syst1}, which hence leads to the contradiction. First, we notice that \eqref{contr9}  ensures that  
\begin{equation}\label{conv}
 \tilde{w}_n \longrightarrow \tilde{w}\equiv 0 \quad \text{in} \ C([-T,T], L^2(\mathbb{R})) , \quad T>0 .
\end{equation}
We now obtain a  local $ L^2$ estimate on $ u_{n,x} $ that is uniform in $n $ by applying  the  local Kato smoothing effect. Let $R\ge 1 $ and $ \varphi \in C^\infty(\mathbb{R}) $ with $ \varphi\equiv 0 $ on $]-\infty,-2R] $, $0\le \varphi'\le 1 $ on $\mathbb{R}$,  $\varphi '\equiv 0 $ on $]-2R,2R[^c $ , $ \varphi'\equiv 1 $ on $[-R,R] $
 and $ |\varphi^{'''} |\le 4 $ on $\mathbb{R} $. Note in particular that $ 0\le \varphi \le 4 R $ on $\mathbb{R} $. Multiplying \eqref{eqtw} by $\varphi \tilde{w}_n $ and integrating by parts, we get 
\begin{align}
\frac{1}{2}\frac{d}{dt} \int \varphi \tilde{w}_n^2 &= - \frac{3}{2} \int_{\R} \varphi' \partial_x \tilde{w}_n^2+ \frac{1}{2} \int_{\R} \varphi^{'''} \tilde{w}_{n}^2
+ \gamma_n \Bigl(\int_{\R} \varphi '   \tilde{w}_n \H \partial_x \tilde{w}_n+ \int_{\R} \varphi   \partial_x  \tilde{w}_n \H \partial_x \tilde{w}_n  \Bigr)
\nonumber \\
& -\dot{x}_n \int_{\R} \varphi' \tilde{w}_n^2 - (c_n-\dot{x}_n) \int_{\R} \varphi \, Q_{\gamma_n,c_n}'  \tilde{w}_n + \frac{\delta_n}{3} \int_{\R} \varphi '
\tilde{w}_n^3 \ . \label{h1}
\end{align}
We notice that 
$$
\Bigl|\int_{\R} \varphi '   \tilde{w}_n \H \partial_x \tilde{w}_n+ \int_{\R} \varphi   \partial_x  \tilde{w}_n \H \partial_x \tilde{w}_n  \Bigr|
\lesssim R \,  \|\tilde{w}_n\|_{H^1}^2 
$$
and that
\begin{align*}
\Bigl| \frac{1}{2} \int_{\R} \varphi^{'''} \tilde{w}_{n}^2+\dot{x}_n \int_{\R} \varphi' \tilde{w}_n^2 &- (c_n-\dot{x}_n) \int_{\R} \varphi \, Q_{\gamma_n,c_n}'  \tilde{w}_n + \frac{\delta_n}{3} \int_{\R} \varphi '
\tilde{w}_n^3 \Bigr| \\
& \lesssim  \|\tilde{w}_n\|_{L^2}^2+ R   \|\tilde{w}_n\|_{L^2}+\delta_n \|\tilde{w}_n\|_{L^2}^{5/2}  \|\tilde{w}_n\|_{H^1}^{1/2} \ .
\end{align*}
Therefore,  integrating \eqref{h1} on $ ]-1,1[ $ and using \eqref{conv} and that $ \gamma_n\to 0 $, we eventually get 
\begin{equation}\label{convo2}
\|\tilde{w}_{n,x} \|_{L^2(]-1,1[\times ]-R,R[)} \tendsto{n\to \infty} 0 \; .
\end{equation}
Combining this with \eqref{renorm2}, we obtain that for any $\delta\in ]-1,1[ $ and any $ \varepsilon>0 $, there exists $ n_{\varepsilon,\delta} \ge 0 $ such that $\forall n\ge n_{\varepsilon,\delta} $,
$$
\int_{0}^\delta \| \tilde{w}_{n,x}(t) \|_{L^2}^2\, dt < \varepsilon \; .
$$
Therefore, there exists  $ n_\delta\ge 0 $ such that for all $n\ge n_\delta$ there exists $ t_{n,\delta}\in [0,\delta] $ with 
\begin{equation}\label{limit3}
\| \tilde{w}_{n,x} ( t_{n,\delta})\|_{L^2} <1/8 \; .
\end{equation}
Let us now use  the energy of the Benjamin equation. We set 
$$
E_n(u)= \frac{1}{2} \int_{R} u_x^2 - \frac{\gamma}{2} \int_{R} u \H u_x -\frac{\delta_n}{6} \int_{\R} u^3 
$$
and seek an estimate on 
$$
 \frac{d}{dt} E_n(\tilde{w}_n))= \int_{\R} \tilde{w}_{n,x} \partial_t \tilde{w}_{n,x}-\gamma \int_{\R}\partial_t  \tilde{w}_n \H  \tilde{w}_{n,x}
 + \frac{\delta_n}{2} \int_{\R}  \tilde{w}_{n}^2 \partial_t  \tilde{w}_n \ .
 $$
 Setting 
 $$ \Theta_n= -\partial^3_x \tilde{w}_n - \gamma_n \H \partial_x^2 \tilde{w}_n - \frac{\delta_n}{2} \partial_x(\tilde{w}_n^2) \ ,$$
it follows from the energy conservation for the Benjamin equation that 
$$
\int_{\R} \tilde{w}_{n,x} \Theta_{n,x}-\gamma \int_{\R}\Theta_n \H  \tilde{w}_{n,x}
 + \frac{\delta_n}{2} \int_{\R}  \tilde{w}_{n}^2\Theta_n = 0 \ .
$$
It thus remains to get an estimate on 
\begin{equation}\label{to}
\Bigl|\int_{\R} \tilde{w}_{n,x} \Delta_{n,x}-\gamma_n \int_{\R}\Delta_n \H  \tilde{w}_{n,x}
 + \frac{\delta_n}{2} \int_{\R}  \tilde{w}_{n}^2\Delta_n\Bigr| 
\end{equation}
where 
$$
 \Delta_n=  \dot{x}_n  \partial_x \tilde{w}_n- (c_n-\dot{x}_n) Q_{\gamma_n,c_n}'- \partial_x ( \tilde{w}_n  Q_{\gamma_n,c_n}) \; .
 $$
 We evaluate separately the contribution of the three terms in \eqref{to}.
 By integration by parts and H\"older inequality, we have 
 \begin{align*}
 \Bigl|\int_{\R} \tilde{w}_{n,x} \Delta_{n,x}\Bigr| & = \Bigl|-  (c_n-\dot{x}_n)\int_{\R}  Q_{\gamma_n,c_n}^{''}  \tilde{w}_{n,x}
  -\frac{3}{2} \int_{\R} Q_{\gamma_n,c_n}'  \tilde{w}_{n,x}^2+\frac{1}{2} \int_{\R} Q_{\gamma_n,c_n}^{'''} \tilde{w}_{n}^2\Bigr|\\
 &  \lesssim (1+\| \tilde{w}_{n}\|_{H^1}) \| \tilde{w}_{n}\|_{H^1}\lesssim 1 \; ,
 \end{align*}
  \begin{align*}
 \Bigl|\int_{\R}  \Delta_{n}\H \tilde{w}_{n,x}\Bigr| & = \Bigl|-  (c_n-\dot{x}_n)\int_{\R}  Q_{\gamma_n,c_n}^{'} \H  \tilde{w}_{n,x}
  -\int_{\R} (Q_{\gamma_n,c_n} \tilde{w}_{n,x} +Q_{\gamma_n,c_n}' \tilde{w}_{n}) \H \tilde{w}_{n,x} \Bigr|\\
 &  \lesssim (1+\| \tilde{w}_{n}\|_{H^1}) \| \tilde{w}_{n}\|_{H^1}\lesssim 1 \; ,
 \end{align*}
 and 
  \begin{align*}
 \Bigl|\int_{\R} \tilde{w}_{n}^2 \Delta_{n}\Bigr| & = \Bigl|-  (c_n-\dot{x}_n)\int_{\R}  Q_{\gamma_n,c_n}^{'}  \tilde{w}_{n}^2
  -\frac{2}{3} \int_{\R} Q_{\gamma_n,c_n}'  \tilde{w}_{n}^3\Bigr|\\
 &  \lesssim (1+\| \tilde{w}_{n}\|_{H^1}) \| \tilde{w}_{n}\|_{H^1}^2\lesssim 1 \; .
 \end{align*}
Gathering all these estimates, we infer that $\displaystyle  |\frac{d}{dt} E_n(\tilde{w}_n(t))| \lesssim 1 $ and thus in view of \eqref{conv}  for $ t\in [0,1] $, we get 
\begin{align*}
\Bigl| \int_{\R} \tilde{w}_{n,x}^2(t,x) -\tilde{w}_{n,x}^2(0,x) \, dx\Bigr| & \lesssim t + \sup_{t\in \R} \Bigl( |\gamma_n|
\Bigl|  \int_{\R} \tilde{w}_{n} \H \tilde{w}_{n,x}\Bigr|  +\delta_n \Bigl| \int_{\R} \tilde{w}_{n}^3 \Bigr| \Bigr) \\
& \lesssim t + \varepsilon(n) (1+\| \tilde{w}_{n}\|_{H^1}) \lesssim t + \varepsilon(n)\; ,
\end{align*}
with $ \varepsilon(n) \rightarrow 0 $ as $ n\to +\infty$.  In view of  \eqref{limit3} and \eqref{conv},  this ensures that 
$$
\| \tilde{w}_{n}(0)\|_{H^1}^2 \le \frac{1}{8}+C \, \delta + o(1) \; ,
$$
for some $ C>0 $.
Taking $\delta>0 $ small enough, this contradicts the second line in \eqref{syst1} and thus completes the proof of Theorem \ref{NL}. \hfill $\square$
\section{Asymptotic Stability}
Fix $ |\gamma|<\gamma_0 $ where $ \gamma_0 $ is given by Proposition \ref{der}. Let  $u_0\in H^1(\R) $ such that
\begin{equation}\label{init}
\|u_0-Q_{\gamma,1}\|_{H^1} < \eta 
\end{equation}
so that the orbital stability result forces 
\begin{equation}\label{orbit}
\sup_{t\in\R}\inf_{y\in\R}  \|u(t)-Q_{\gamma,1}(\cdot-y)\|_{H^1} < \varepsilon=\varepsilon(\eta) .
\end{equation}
 We set 
\begin{equation}\label{deftalpha}
\alpha=\limsup_{t\to \infty} \|u(t) \|_{L^2(x>t/2)} .
\end{equation}
The conservation of the $ L^2$-norm of the solution, \eqref{orbit}, the decay of $Q_{\gamma,1}$ and Lemma \ref{parameter} (especially \eqref{decomp4}), together force $ |\alpha-\|Q_{\gamma,1}\|_{L^2}| \lesssim \varepsilon $. 
We take $\eta $ small enough in \eqref{init} so that \eqref{tr2} then implies that there exists  a unique $ c_*>0 $ with $ |c_*-1|\lesssim \varepsilon $ such that 
\begin{equation}\label{deftc}
 \|Q_{\gamma,c_*} \|_{L^2}^2= \alpha^2 \; .
\end{equation}
 It then follows from \eqref{smooth} and \eqref{init} that 
\begin{equation}\label{tt}
\|Q_{\gamma,1}-Q_{\gamma,c_*}\|_{H^1} \lesssim  \varepsilon \ .
\end{equation}
Let $ \varepsilon_0 >0 $ be the constant interfering in the hypothesis \eqref{hyp1} of Theorem \ref{NL}. Taking $ \eta >0 $ small enough in \eqref{init},   the modulation results  together with \eqref{tt} ensure that there exists $ x \, :\; \R \to \R $ of class $ C^1 $  such that for all $t\in\R $,
\begin{align}
\|u(t, \cdot+x(t)) -Q_{\gamma,1}\|_{H^1} <\varepsilon_0 \ ,\label{bound1} \\
 \quad \int_{\R} u(t, \cdot+x(t)) Q_{\gamma,c_*}' =0 \ , \label{ortho1}\\
 |\dot{x}(t)-1|\ll 1 \ . \label{bound2} 
\end{align}
Therefore, $ u $ satisfies the hypotheses of the almost $H^1$-monotonocity proposition (Proposition \ref{H1Monot}) and the same arguments as the ones after \eqref{tftf} in the proof of Proposition \ref{propdecay} lead to the uniform decay at the right of the $H^1$-norm of $ u $, i.e. 
\begin{equation}\label{bound3}
\lim_{R\to +\infty}\limsup_{t\to +\infty} \|u(t, \cdot+x(t))\|_{H^1(x>R)} =0\; .
\end{equation}
%and the almost $L^2$-monotonicity (Proposition \ref{L2Monot} a preciser) together with \eqref{deftalpha} force that actually 
%\begin{equation}\label{deftalpha}
%\lim_{t\to \infty} \|u(t) \|_{L^2(x>t/2)} =\alpha \; .
%\end{equation}
We will now state  a  proposition which ensures that from any sequence of values of $ u $ along an increasing sequence of time going to infinity we can extract an asymptotic object that satisfies the hypotheses of the nonlinear Liouville theorem. This proposition together with the  nonlinear Liouville theorem will lead to the asymptotic stability result. However, to establish this proposition we  need to use the continuity of the flow-map associated with the Benjamin equation with respect to the $ H^1$-weak topology. This is the aim of the following lemma that is proven in the appendix.
\begin{lemma}\label{weakcontinuity}
Let $ u_0\in H^1(\R) $ and $ \{u_{0,n}\}_{n\ge 0} \subset H^1(\R) $ such that $ u_{0,n} \rightharpoonup u_0 $ in $ H^1(\R) $. Then, denoting by $ u\in C(\R;H^1(\R))  $ and $ u_n \in C(\R;H^1(\R))$, $ n\ge 0 $, the solutions to \eqref{MainEq} emanating respectively from  $ u_0 $ and $ u_{0,n} $, $n\ge 0 $, for any $ T>0 $ it holds\footnote{This convergence means that for any $ \phi \in H^1(\R) $, $(u_n,\phi)_{H^1} \to (u,\phi)_{H^1} $ in $C([-T,T])$.} 
\begin{equation}\label{cvfaible}
u_n \rightharpoonup u\in C_w([-T,T];H^1(\R)) \; .
\end{equation}

\end{lemma}
\begin{proposition}\label{propro}
For any  sequence $\{t_{n}\} $ with $ t_n \nearrow +\infty $, there exists  a subsequence $ \{t_{n_k}\} $ of $\{t_n\} $ and $ \tilde{u}_0\in H^1(\R) $ such that for any $ T>0 $, 
\begin{equation} \label{e2}
[t\mapsto u(t_{n_k}+t, \cdot-x(t_{n_k}+t))] \rightharpoonup [t\mapsto\tilde{u}(t,\cdot+\tilde{x}(t))]   \quad \text{in} \; C_w([-T,T];H^1(\R)) 
\end{equation}
and for any $ R>0 $ and any $ t\in\R $,
\begin{equation}\label{bound10}
\lim_{k\to \infty} \|u(t_{n_k} +t, \cdot + x(t_{n_k}+t)) - \tilde{u}(t,\cdot +\tilde{x}(t))\|_{L^2([-R,+\infty[)} =0 \;,
\end{equation}
where $ \tilde{u} \in C(\R;H^1(\R))$ is the solution to \eqref{MainEq} emanating from $\tilde{u}_0 $. Moreover, $ \tilde{u} $ satisfies the hypotheses of the nonlinear Liouville theorem (Theorem \ref{NL}).
\end{proposition}
\begin{proof}
Let $ t_n \nearrow +\infty $. In view of \eqref{bound1}, there exists a subsequence $ t_{n_k} \nearrow +\infty $  and $ \tilde{u}_0 \in H^1(\R) $ such that 
\begin{equation} \label{e1}
u(t_{n_k}, \cdot+x(t_{n_k})) \rightharpoonup \tilde{u}_0  \quad \text{in} \; H^1(\R) 
\end{equation}
Moreover, in view of \eqref{bound2}, $ \{x(t_n+\cdot)-x(t_n)\} $ is uniformly equi-continuous and thus Arzela-Ascoli Theorem ensures that  there exists $ \tilde{x}\in C(\R) $ such that for all $ T>0 $, we can ask that
   \begin{equation}\label{cvx}
   x(t_{n_k}+\cdot)-x(t_{n_k}) \tendsto{t\to+\infty} \tilde{x} \mbox{ in } C([-T,T]) \; .
   \end{equation}
Let $\tilde{u} \in C(\R;H^1(\R)) $ be the solution to \eqref{MainEq} emanating from $ \tilde{u}_0 $. \eqref{e2} then follows from  \eqref{e1}, \eqref{cvx} and  Lemma \ref{weakcontinuity}. Moreover, by the local compactness of $ H^1(\R) $ in $ L^2(\R) $, for all $t\in\R $, we can also ask that 
\begin{equation} \label{e3}
u(t_{n_k}+t, \cdot-x(t_{n_k}+t)) \rightarrow \tilde{u}(t,\cdot+\tilde{x}(t))   \quad \text{in} \;L^2_{loc}(\R) \; . 
\end{equation}
Note that \eqref{e2} together with \eqref{bound1} -\eqref{ortho1}  ensure that  $ \tilde{u} $ satisfies 
\begin{equation} \label{bound11}
\sup_{t\in\R} \|\tilde{u}(t, \cdot+\tilde{x}(t))-Q_{\gamma,1}\|_{H^1} \le \varepsilon_0 
\quad \text{and}\quad  \int_{\R} \tilde{u}(t, \cdot+\tilde{x}(t)) Q_{\gamma,c_*}' =0
\end{equation}
 that is the hypothesis \eqref{hyp1} of Theorem \ref{NL}. In particular from the modulation result, by uniqueness, we deduce that $ \tilde{x} $ is of class $ C^1 $ and satisfies 
\begin{equation}\label{mod2}
\sup_{t\in\R} |1-\dot{\tilde{x}}(t)| \ll 1 \; .
\end{equation}
Moreover  \eqref{e2} together 
with \eqref{bound3} imply that  
\begin{equation}\label{bound4}
\lim_{R\to +\infty} \sup_{t\in\R} \|\tilde{u} (t, \cdot+\tilde{x}(t))\|_{H^1(x>R)} =0  \; .
\end{equation}
This proves the $ H^1 $-uniform decay at the right of the asymptotic object $ \tilde{u} $.

Now, we want to prove that the convergence in \eqref{e3} is actually uniform on compact intervals. For this, we first notice that since by \eqref{bound1}, the trajectory of $u $ is bounded in $ H^1(\R) $, it follows from \eqref{MainEq} that $ \sup_{t\in\R}  \|\partial_t u (t)\|_{H^{-2}} <\infty $. Therefore, $ t\mapsto  u(t) $ is uniformly continuous with values in $ H^{-2}(\R)$ and thus by interpolation with the $ H^1$-bound, uniformly continuous with values in $ L^2(\R) $ too. This ensures that the sequence of functions $ (f_k)_{k\ge 0}$ defined by  $ f_k\, :\, t\mapsto u(t_{n_k}+t, \cdot-x(t_{n_k}+t))$ is equi-continuous with values in $ L^2(\R) $ and thus the convergence in \eqref{e3} is actually uniform on compact interval, i.e. 
\begin{equation} \label{e33}
[t\mapsto u(t_{n_k}+t, \cdot-x(t_{n_k}+t))] \rightarrow [t\mapsto \tilde{u}(t,\cdot+\tilde{x}(t))]   \quad \text{in} \;C([-T,T];L^2_{loc}(\R)) \; . 
\end{equation}
 This last inequality combined with \eqref{bound3}  and \eqref{bound4} ensure that for any $ T>0 $ and $ R>0 $, 
 \begin{equation}\label{bound6}
\lim_{k\to +\infty}  \sup_{t\in [-T,T]} \|u(t_{n_k}+t , \cdot + x(t_{n_k}+t))-\tilde{u}(t, \cdot+\tilde{x}(t))\|_{L^2([-R,+\infty[)} =0  \; .
 \end{equation}
We move now to prove the $ H^1$-uniform decay at the left of the asymptotic object $ \tilde{u}$.  For this, we need to define the following functionals :  For $ \theta_0\ge 4 $ as in \eqref{nono1}, we define $H^R_r $ and $H^R_{l} $ on $ H^1(\R) $
\begin{equation}\label{defHR}
H^{R,r}(u)=\int_{\R} (\frac{1}{2} u_x^2 -\frac{\gamma}{2} u \H u_x -\frac{1}{6} u^3+ \theta_0 u^2) \chi\Bigl(\frac{\cdot+2R}{R^{3/4}}\Bigr)
\end{equation}
and 
$$
H^{R,l}(u)=\int_{\R} (\frac{1}{2} u_x^2 -\frac{\gamma}{2} u \H u_x -\frac{1}{6} u^3+ \theta_0 u^2) \Bigl(1-\chi\Bigl(\frac{\cdot+2R}{R^{3/4}}\Bigr)\Bigr)
$$
Note that, on  account of \eqref{estgg} it holds 
$$
\Bigl| \int_{\R} u \H u_x \chi\Bigl(\frac{\cdot+2R}{R^{3/4}}\Bigr)\Bigr)\Bigr|
\le \int_{\R}  (\frac{1}{4} u_x^2+2 u^2) \Bigl(1-\chi\Bigl(\frac{\cdot+2R}{R^{3/4}}\Bigr)\Bigr) + c R^{-3/4} 
$$
and thus as in \eqref{nono1}  it holds 
\begin{align}
H^{R,l}(u)& \ge\frac{1}{4}  \int_{\R} (u_x^2  +u^2) \Bigl(1-\chi\Bigl(\frac{\cdot+2R}{R^{3/4}}\Bigr)\Bigr)-c \gamma R^{-3/4} \nonumber \\ 
&\ge \frac{1}{4} \|u\|_{H^1(-\infty, -2R)}^2-C \gamma R^{-3/4}  \; . \label{de1}
\end{align}
We will make use of the following lemma, proven in the appendix,  that ensures that $ t\mapsto H^{R,r} (u(t, \cdot + x(t)))$ and $ t\mapsto  H^{R,r} (\tilde{u}(t, \cdot + \tilde{x}(t))) $ are almost non increasing functions. This follows from the almost monotonicity propositions (Propositions \ref{L2Monot} and \ref{H1Monot}).  
\begin{lemma}\label{LemdecayH}
Let $u\in C(\R;H^1(\R)) $ be a solution to \eqref{MainEq}. 
For $|\gamma|\le 1/2$ and under the same assumptions as in Proposition \ref{L2Monot}, there exist $C>0$ only depending on $ M(u_0)$, $E(u_0) $ and $ R_0 $  such that for all $R>1$ and $(t,t_0)\in \R^2 $ with $ t\ge  t_0$ such that 
\begin{equation}\label{decH}
H^{R,r} (u(t, \cdot + x(t)))\le H^{R,r} (u(t_0, \cdot + x(t_0))) +  C R^{-1/4} \; .
\end{equation}
\end{lemma}
We proceed by contradiction assuming that the $ H^1$-uniform decay does not hold. Then there exists $ \beta_0 >0 $ such that for all $ R>0 $ there exists $ t_R\in \R $ with 
$$
\|\tilde{u}(t_R, \cdot +\tilde{x}(t_R))\|_{H^1(-\infty,-2R)} >\beta_0 \; .
$$
In particular, taking  $ R>0 $ such that $ R^{-3/4} \ll \beta_0 $ , \eqref{de1} leads to 
$$
H^{R,l}(\tilde{u}(t_{R}, \cdot +\tilde{x}(t_{R})))> 2^{-3} \beta_0  \; .
$$
Since  
$$
 H^{R,l}(\tilde{u}(t, \cdot +\tilde{x}(t)) +H^{R,l}(\tilde{u}(t, \cdot +\tilde{x}(t))= \theta_0 M(\tilde{u}) +E(\tilde{u}) = \theta_0 M(\tilde{u}(0))+E(\tilde{u}(0)) ,
 $$
 it follows that 
 $$
 H^{R,r}(\tilde{u}(t_{R}, \cdot +\tilde{x}(t_{R}))) \le   \theta_0 M(\tilde{u}(0)) +E(\tilde{u}(0))-2^{-3} \beta_0 \; .
 $$
 Therefore, Lemma \ref{LemdecayH} ensures that  $\forall  t\ge t_R$ and $ R\ge 1 $ with $ R^{-1/4} \ll \beta_0 $, 
 \begin{equation} \label{limsup1}
  H^{R,r}(\tilde{u}(t, \cdot +\tilde{x}(t)))\le  \theta_0 M(\tilde{u}(0)) +E(\tilde{u}(0))-2^{-4} \beta_0 \; .
 \end{equation}
 On the other hand,  since $ \tilde{u} \in C(\R;H^1(\R)) $ there exists $ R_1 >0$ with $ R_1^{-1/4} \ll \beta_0 $ such that 
 \begin{equation} \label{limsup2}
\inf_{\tau\in [0,1]} H^{R_1,r} (\tilde{u}(\tau, \cdot +\tilde{x}(\tau))) \ge  \theta_0 M( \tilde{u})+E( \tilde{u}) - 2^{-8} \beta_0 \; .
 \end{equation}
Here we may notice that, taking $R=R_1 $ in \eqref{limsup1}, this forces $ t_{R_1}>0 $. However, to exclude the case $ t_{R_1} >0 $ we will apply an argument that works whatever the  sign of $ t_{R_1} $ is. As in the previous section, this argument makes use of the local  Kato smoothing effect in order to prove a strong convergence to zero of 
$$
 w_k(t,\cdot)= u(t_{n_k} +t, \cdot + x(t_{n_k}+t)) - \tilde{u}(t,\cdot +\tilde{x}(t))
 $$  in $ H^1(-R,+\infty )$ along  some sequences of times.
 It is direct to check that $ w_k $ satisfies the following equation 
\begin{equation} \label{eqw}
 \partial_t w_k + \ \partial_x^3 w_k  +  \gamma \, \H  \partial_x^2 w_k + \frac{1}{2} \partial_x(z_k w_k)  = g_k
 \end{equation}
where 
$$
z_k(t,x)=u(t_{n_k} +t, \cdot + x(t_{n_k}+t)) + \tilde{u}(t,\cdot +\tilde{x}(t))
$$
and 
$$
g_k(t,x)= \dot{x}(t_{n_k}-t) u_x(t_{n_k} +t, \cdot + x(t_{n_k}+t))-\dot{\tilde{x}}(t)    \tilde{u}_x(t,\cdot +\tilde{x}(t))
$$
As in the previous section, we now  establish a  local $ L^2$ estimate on $\partial_x  w_{k}$. Let  $ \varphi \in C^\infty(\mathbb{R}) $ with $ \varphi\equiv 0 $ on $]-\infty,-2] $, $0\le \varphi'\le 1 $ on $\mathbb{R}$,  $\varphi '\equiv 0 $ on $]-2,2[^c $ , $ \varphi'\equiv 1 $ on $[-1,1] $
 and $ \|\varphi'\|_{W^{2,\infty}} \le 4 $ on $\mathbb{R} $.
 
  We set $ \varphi_R= R \,\varphi(\cdot/R) $.
 Note in particular that $ 0\le \varphi_R\le 4 R $ on $\mathbb{R} $ and $ \varphi_R'\equiv 1 $ on $[-R,R] $ and that the same type of considerations as the ones to obtain \eqref{deriv2} for $ 1\le k \le 3$  lead to 
 \begin{equation}\label{derder1}
 \|D_x^{k} \varphi_R\|_{L^\infty_x} \lesssim R^{1-k} \; .
 \end{equation}
  Multiplying \eqref{eqw} by $\varphi_R w_k $ and integrating by parts, we get 
\begin{align}
\frac{1}{2}\frac{d}{dt} \int \varphi_R w_k^2 &= - \frac{3}{2} \int_{\R} \varphi_R' (\partial_x w_k)^2+ \frac{1}{2} \int_{\R} \varphi_R^{'''}w_k^2
+ \gamma \Bigl(\int_{\R} \varphi_R '   w_k \H \partial_x w_k+ \int_{\R} \varphi_R   \partial_x  w_k \H \partial_x w_k \Bigr)
\nonumber \\
& -\int_{\R} (z_{k,x} \varphi_R+\frac{1}{2} z \varphi_R') w^2_k +\int_{\R}  \varphi_R g w_k  \; . \label{hh1}
\end{align}
According to \eqref{bound6}, we know that for all $t\in\R $, $\displaystyle  \lim_{k\to \infty}\|w_k(t)\|_{L^2(\text{supp} \,\varphi_R)}=0 $.
We thus infer that 
$$
\lim_{k\to +\infty} \Bigl( \int_{\R} (|\varphi_R|+|\varphi_R^{'''}|) w_k^2+|\int_{\R} \varphi_R '   w_k \H \partial_x w_k|
+|\int_{\R} (z_{k,x} \varphi_R+\frac{1}{2} z \varphi_R') w^2_k|+|\int_{\R}  \varphi_R g w_k| \Bigr) =0 \;.
$$
where we used that by Sobolev imbedding and interpolation 
$$
|\int_{\R} (z_{k,x} \varphi_R w^2_k|\lesssim R  \|z_k\|_{H^1} \|w_k\|_{L^4_{(\text{supp} \varphi_R)}}^2 
\lesssim R  \|z_k\|_{H^1} \|w_k\|_{H^1}^{1/2}\|w_k\|_{L^2_{(\text{supp} \,\varphi_R)}}^{3/2}\tendsto{k\to +\infty} 0 \; .
$$
Therefore, for $ T\in \R $ to be chosen later,  integrating \eqref{hh1} on $ [T,T+1] $ and making use of the Lebesgue Dominated Convergence Theorem we get 
$$
\int_{T}^{T+1} \int_{-R}^R (\partial_x w_k)^2 \le  \epsilon_R(k) + \Bigl| \int_{\R} \varphi_R   \partial_x  w_k \H \partial_x w_k\Bigr| 
$$
where $ \epsilon_R(k)\to 0 $ as $ k\to 0 $. 
We do not know if the last term in the above estimate tends to $ 0 $ with $ k$. However, making use of   \eqref{tgtg}, \eqref{besov1}, and \eqref{derder1} we obtain 
$$
\Bigl| \int_{\R} \varphi_R   \partial_x  w_k \H \partial_x w_k\Bigr| =
\frac{1}{2} \Bigl| \int_{\R}     w_k \partial_x( [\H,\varphi_R] \partial_x w_k)\Bigr|\lesssim 
R^{-3/4} 
$$
which  ensures that 
\begin{equation}\label{bound7}
\lim_{R\to +\infty} \limsup_{k\to +\infty} \int_{T}^{T+1} \int_{-R}^R (\partial_x w_k)^2 =0  \; .
\end{equation}
In particular, for any $ R>0 $,   by a diagonal process we can construct a   sequence $( \tau_{n_k})\subset ]T,T+1[ $ such that 
\begin{equation}\label{bound8}
\lim_{k\to\infty}  \int_{-R}^R (\partial_x w_k(\tau_{n_k}))^2=0 \; .
\end{equation}
Combining this last convergence result with \eqref{bound3}  and \eqref{bound4}, we infer that 
$$
\lim_{k\to \infty}   \int_{-R}^{+\infty}  (\partial_x w_k(\tau_{n_k}))^2=0 \; 
$$
which together with \eqref{bound6}  lead to 
\begin{equation}\label{bound9}
\lim_{k\to \infty}\|u(t_{n_k}+\tau_{n_k}, \cdot + x(t_{n_k}+\tau_{n_k})) - \tilde{u}(\tau_{n_k},\cdot +\tilde{x}(\tau_{n_k}))\|_{H^1([-R,+\infty[)} =0 \; .
\end{equation}

Now, taking $ T=0 $ in \eqref{bound7}, \eqref{bound9} with $ R=2 R_1$ together with  \eqref{limsup2}  ensure that 
$$
\limsup_{t\to \infty} H^{R_1,r} (u(t, \cdot +x(t))) \ge  \theta_0 M( \tilde{u})+E( \tilde{u}) - 2^{-8} \beta_0 \; .
$$
On the other hand, taking $T=T_{R_1}+1$ in \eqref{bound7}, \eqref{bound9} with $ R=2 R_1$ together with  \eqref{limsup1}  ensure that 
$$
\limsup_{t\to \infty} H^{R_1,r} (u(t, \cdot +x(t))) \le  \theta_0 M( \tilde{u})+E( \tilde{u}) - 2^{-4} \beta_0 
$$
which leads the contradiction and hence proves the $H^1$-uniform decay at the left of $ \{\tilde{u}(t,\cdot+\tilde{x}(t)), \; t\in\R \} $.
\end{proof}
Let us now conclude the proof of the asymptotic stability. Let $\{t_n\} \subset \R_+ $ such that $ t_n\nearrow +\infty $. The preceding proposition combined with the nonlinear Liouville theorem ensure that there exist a subsequence $\{t_{n_k}\} $ of $\{t_n\} $, $ c>0 $ and $y\in\R $ with $ |c-1|\ll 1 $ such that for any $ R>0 $\begin{equation}\label{bound100}
\lim_{k\to \infty} \|u(t_{n_k}  \cdot + x(t_{n_k})) - Q_{\gamma,c}(\cdot +y)\|_{L^2(]-R,+\infty[)} =0 \;,
\end{equation}
We claim that \eqref{bound100} forces $ c=c_* $. Indeed , first combining \eqref{bound100} with  \eqref{deftalpha} and  \eqref{bound2}, we infer  that $ \|Q_{\gamma,c}\|_{L^2}\le \alpha $. Second, assuming that $ \|Q_{\gamma,c}\|_{L^2}< \alpha $, \eqref{bound100} would force, for any $ R>0 $,
$$
\limsup_{k\to \infty} \|u(t_{n_k}  \cdot + x(t_{n_k}))\|_{L^2(]-R,+\infty[)} <\alpha ;
$$
but then taking $ \vartheta=3/8<1/2 $ in  \eqref{defI2}, the almost monotonicity of $ t\mapsto I^{-2R}_{t_{n_k}}(t)$  would imply 
$$
\limsup_{t\to +\infty} \|u(t, \cdot + x(t)\|_{L^2(x> t/2)} <\alpha
$$
that contradicts \eqref{deftalpha}. Therefore, $ \|Q_{\gamma,c}\|_{L^2}= \alpha $ that forces $c=c_* $ by the uniqueness of such $ c$ in the vicinity of $ 1$. Now combining \eqref{tt}, \eqref{bound1} and \eqref{bound10}, we infer that $\|Q_{\gamma,c^*}-Q_{\gamma,c^*}(\cdot+y)\|_{H^1} \ll 1 $ which forces $ |y| \ll 1 $.  But then \eqref{ortho1} together with the uniqueness given by the Implicit Function Theorem in the proof of the modulation lemma (Lemma \ref{parameter}) directly force $ y=0 $. This proves that  the only possible value for $ \tilde{u}_0  $ in Proposition \ref{propro} is $ \tilde{u}_0=Q_{\gamma,c_*}$. This ensures that actually 
\begin{equation}\label{kj11}
u(t, \cdot + x(t)) \rightharpoonup Q_{\gamma,c_*} \quad \text{in} \; H^1(\R) 
\end{equation}
 and for any $ R>0 $, 
\begin{equation}\label{kj1}
\lim_{t\to \infty} \|u(t, \cdot + x(t)) - Q_{\gamma,c_*}\|_{L^2([-R,+\infty[)} =0 \; .
\end{equation}
Now we observe that this last convergence result together with  \eqref{bound2} ensures that 
$$
\Bigl(u(t) \, ,\, Q_{\gamma,c_*}(\cdot -x(t))\Bigr)_{L^2(]t/2,+\infty[)} \tendsto{t\to +\infty} \|Q_{\gamma,c_*}\|_{L^2_x}^2=\alpha^2 \; .
$$
Combining this  limit with \eqref{deftalpha} we obtain that
\begin{equation}\label{end}
\limsup_{t\to+\infty} \|u(t)-Q_{\gamma,c_*}(\cdot-x(t))\|_{L^2(]t/2,+\infty[)}^2\le 2\alpha^2-2\alpha^2=0 \; .
\end{equation}
To prove the convergence in the $H^1$-level, we use again the Kato smoothing effect this time on $ w(t)=u(t,\cdot+x(t))-Q_{\gamma,c^*} $. Proceeding as above for obtaining \eqref{bound7}, making use this time of \eqref{kj1}, we obtain that for any $ (t_n)\nearrow +\infty  $ and any $ \beta>0 $ there exists a subsequence $(t_{n_k}) $ such that 
\begin{equation}\label{bound77}
\lim_{R\to +\infty} \limsup_{k\to +\infty} \int_{t_{n_k}-\delta}^{t_{n_k}} \int_{-R}^R w_x^2 =0  \; .
\end{equation}
By uniqueness of the possible limit, this holds actually for the whole sequence $ (t_n) $ and thus for  any $ (t_n)\nearrow +\infty  $, any $ R>0 $ and any $ \beta>0 $ there exists $(\tau_n) \in [t_n-\beta,t_n] $ such that 
\begin{equation}\label{bound88}
\lim_{n\to\infty}  \int_{-R}^R w_x(\tau_n)^2=0 \; .
\end{equation}
Combining this with \eqref{bound3} and \eqref{kj1},  we obtain that 
$$
\lim_{n\to \infty}  \int_{-2R}^{+\infty} (w(\tau_n)^2+w_x(\tau_n)^2)=0 \Rightarrow \lim_{n\to \infty} H^{R,r}(u(\tau_n, +x(\tau_n))= H^{R,r}(Q_{\gamma,c^*}) \; .
$$
with $H^{R,r}(\cdot) $  defined in \eqref{defHR}. Now, according to \eqref{finalder} and \eqref{final H1} we get 
$$
\sup_{t\in [\tau_n, \tau_n+\beta]} H^{R,r}(u(t, \cdot+x(t))\le H^{R,r}(u(\tau_n, \cdot+x(\tau_n)) + C \beta \; .
$$
Since this holds for any $\beta>0  $  and ($|t_n- \tau_n|\le \beta  $, $\forall n\in \N$),  it follows that
$$
 \limsup_{n\to \infty} H^{R,r}(u(t_n, +x(t_n))\le H^{R,r}(Q_{\gamma,c^*}) \; .
$$
This inequality together with  the convergence results  \eqref{kj11} and \eqref{kj1} force 
$$
 \limsup_{n\to \infty} \int_{R} u_x(t_n, \cdot+x(t_n))^2 \chi(\frac{\cdot+2R}{R^{3/4}})\le \int_{R}  (\partial_x Q_{\gamma,c^*})^2 \chi(\frac{\cdot+2R}{R^{3/4}})
 $$
This last  estimate combined with \eqref{kj11} ensures that 
$$
\lim_{t\to+\infty} \|u_x(t_n, \cdot+x(t_n))-Q_{\gamma,c^*}\|_{H^1(-R,+\infty)} =0 
$$
and thus, by uniqueness of the possible limit together with \eqref{kj1} 
\begin{equation}\label{kj2}
\lim_{t\to \infty} \|u(t) - Q_{\gamma,c_*}(\cdot-x(t))\|_{H^1(]-R+x(t),+\infty[)} =0 \;.
\end{equation}
To prove that the convergence in $ H^1 $ holds actually in $ [t/2,+\infty[ $, we set 
$$
K^{-R}_{t_0}:=\theta_0 I^{-2R}_{t_0} + J^{-2R}_{t_0}
$$
where we take $\vartheta=\frac{3}{8} $ and $ \theta_0 \ge 4 $ as in \eqref{nono1}.
We thus get 
\begin{align*}
K^{-R}_{t_0}(w(t))& = K^{-R}_{t_0}(u(t,\cdot))+ K^{-R}_{t_0}(Q_{\gamma,c_*}(\cdot-x(t)))+\theta_0 \int_{\R} u(t) Q_{\gamma,c_*}(\cdot-x(t))  \tilde{\Psi}(t)\\
& +\int_{\R}  \partial_x Q_{\gamma,c_*}(\cdot-x(t)) u_x(t)   \tilde{\Psi}(t) -\frac{1}{2} \int_{\R} \Bigl(u^2 Q_{\gamma,c_*}(\cdot-x(t))+u Q_{\gamma,c_*}(\cdot-x(t))^2\Bigr) \tilde{\Psi}(t)\\
&-\frac{\gamma}{2} \int_{\R} \Bigl(u(t) \H  \partial_x Q_{\gamma,c_*}(\cdot-x(t))+
Q_{\gamma,c_*}(\cdot-x(t)) \H u_x\Bigr) \tilde{\Psi}(t)\\
&= K^{-R}_{t_0}(u(t,\cdot))+ K^{-R}_{t_0}(Q_{\gamma,c_*}(\cdot-x(t)))+D(t) \ .
\end{align*} 
According to the convergence results \eqref{kj11} and \eqref{kj2}, we have 
$$
\lim_{t\to +\infty} D(t)= \theta_0  M(Q_{\gamma,c^*})+ E(Q_{\gamma,c^*}) \; .
$$
Let $\beta>0 $. For $ R \ge 1 $  such that $ R^{-1/4}\ll \beta$, Lemma \ref{LemdecayH} leads to
$$
K^{-R}_{t_0}(w(t))\le K^{-R}_{t_0}(u(t_0,\cdot))-(\theta_0 M(Q_{\gamma,c^*})+ E(Q_{\gamma,c^*})) +\beta/4 \; .
$$
Therefore taking $(t_0,R) $ large enough such that 
$$
|K^{-R}_{t_0}(u(t_0,\cdot))-\theta_0 M(Q_{\gamma,c^*})+ E(Q_{\gamma,c^*})|<\beta/8 ,
$$
that is possible thanks to \eqref{kj2}, we get 
$$
K^{-R}_{t_0}(w(t))<3\beta/8 \; .
$$
Since for $ t >t_0 $ large enough it holds  $ x(t_0)-2R+ R^\frac{3}{4} +\frac{3}{8} (t-t_0) <t/2 $, it follows from \eqref{Psit2} and \eqref{nono1} that 
$$
\int_{t/2}^{\infty} (w^2(t) +w_x^2(t))\le \int_{\R} (w^2(t) +w_x^2(t)) \chi\Bigl(\frac{\cdot-(x(t_0)-2R + \frac{3}{8} (t-t_0))}{(R+\frac{1}{8}(t-t_0))^{3/4}}\Bigr) <\beta 
$$
which proves that 
\begin{equation}\label{kj22}
\lim_{t\to \infty} \|u(t) - Q_{\gamma,c_*}(\cdot-x(t))\|_{H^1(]t/2,+\infty[)} =0 \;.
\end{equation}
It remains to prove the convergence of $ \dot{x} $ toward $ c_* $ . Since this can be deduced by classical arguments (see for instance \cite{MM2}), we only sketch the proof here. Setting $w(t)=u(t)-Q_{\gamma,c_*}(\cdot-x(t)) $ and proceeding  as in the end of  the proof of the modulation lemma with \eqref{end} in hands,  we end with 
$$
|\dot{x}(t)-c_*|\Bigl(\|Q_{\gamma,c_*}'\|^2_{L^2} + O(\|w(t)\|_{H^1})\Bigr) = g(t)
$$
with $ g(t) \to 0 $ as $ t\to +\infty $, that yields the desired convergence result. This completes the proof of Theorem \ref{Main} in the case $ c=1 $ and $ |\gamma| < \gamma_0 $. 

 Lastly, since $u$ is a solution of \eqref{MainEq}  with $ \gamma=\beta\in\R $  if and only if $v(t,x) = \lambda^{2} \ u(\lambda^{3}\ t, \lambda \ x)$ is a solution to \eqref{MainEq} with $ \gamma=|\lambda| \beta $, there is  an equivalence between the asymptotic stability of $Q_{\beta,1}$ and the one of $Q_{|\lambda|\beta,\lambda^2}$ . This leads to the  asymptotic stability of $ Q_{\gamma,c} $, with $ \gamma \in \R^* $, as soon as 
$ c>\gamma_0^{-2} \gamma^2 $.

\section{Appendix}
\subsection{Proof of the commutator estimate \eqref{besov1}}
We first prove the result by replacing $ f $ by $f_R= f \, \eta(\cdot/R) $ where $ \eta $  is a smooth even bump function with $ 0\le \eta\le 1 $ on $\R $, $\eta =1$ on  $[-1,1] $  and $ R\ge 1$.  Clearly $f_R\in W^{3,\infty}(\R) \cap L^2(\R) $. 
We denote by $ P_-$ and $ P_+ $ the Fourier projectors on the negative and positive $ x$-frequencies respectively.
Using that $ \overline{P_-(v)}=P_+(\overline{v}) $ for any $ v\in L^2(\R) $ and that we deal with real-valued functions, we may reduce the expression as 
\begin{eqnarray*}
 [ \H, f] v_x 
& =& 2 \Re \Bigl(  i  P_+(f_R v_x) -i  f P_+ v_x \Bigr) \\
& = &2  \Re \Bigl(  i  P_+(f_R P_- v_x) -i  P_-(f_R P_+ v_x) \Bigr) \\
& = & -4 \Im \Bigl(  P_+(f_R P_- v_x) \Bigr) \ .
\end{eqnarray*}
Therefore, 
$$ 
\|\partial_x ( [ \H, f_R] v_x )\|_{L^2_x} \le 4 \|\partial_x   P_+(f_R P_- v_x)\|_{L^2_x} \ .
$$
Now, we make use of the classical homogeneous Littlewood-Paley decompositon of $ L^2$-functions, given by  $f=\sum_{N>0} P_N f $,
where $ N$ lives  in $\{2^k, \; k\in \Z \} $ and $ P_N $ is a smooth Fourier projector on frequencies of order $ N$. 
Using this  decompositon of $ f $ and $ v $, and taking into account the frequency projection, we eventually get 
\begin{eqnarray*}
\| \partial_x P_+(f_R P_- v_x) \|_{L^2_x} & \lesssim & \sum_{N_1>0} \  \sum_{0<N_2\lesssim N_1}
\| \partial_x P_+(P_{N_1} f_R P_- P_{N_2} v_x) \|_{L^2_x} \\ 
& \lesssim & \sum_{N_1>0} \  \sum_{0<N_2\lesssim N_1} N_1 N_2 \|P_{N_1} f_R \|_{L^\infty_x} \|P_{N_2} v \|_{L^2_x} \\
& \lesssim & \sum_{0 < N_1 <1} \  \sum_{0<N_2 \le 1} N_1^{2-\varepsilon/2} N_2^{\varepsilon/2}  \|P_{N_1} f_R \|_{L^\infty_x} \|P_{N_2} v \|_{L^2_x}\\
& & + \sum_{N_1 \geq 1} \ \sum_{0<N_2\lesssim N_1 } N_1^{2+\varepsilon/2} N_2^{-\varepsilon/2}  \|P_{N_1} f_R \|_{L^\infty_x} \|P_{N_2} v \|_{L^2_x}\\
& \lesssim &  \Bigl(  \|D_x^{2-\varepsilon}f_R\|_{L^\infty_x}+\|D_x^{2+\varepsilon}f_R\|_{L^\infty_x}\Bigr) \|v\|_{L^2_x}
\end{eqnarray*}
which proves  \eqref{besov1} for $ f_R $. The result for   $ f\in W^{3,\infty}(\R) $ follows by passing to the limit as $ R\to +\infty $. Indeed, for $ R\ge 1 $ large enough and $ 0\le s\le 3$, it holds that
$ \|f_R\|_{W^{s,\infty}} \lesssim \|f_R\|_{W^{s,\infty}} ,$ and for all $ v\in H^2(\R)$, 
$$
[\H,f_R]v_x=f_R \H v_x -\H(f_R v_x) \tendsto{R\to +\infty}  f \H v_x -\H(f v_x) =[\H,f]v_x \quad \text{in} \quad H^1(\R) \;  $$
which gives the result for $ f\in W^{3,\infty}(\R) $ and $  v\in H^2(\R) $. The result for $v\in L^2(\R) $ follows by density arguments.
\subsection{Proof of Lemma  \ref{weakcontinuity}}
We give a simple proof based on the approach given in \cite{GoubetMolinet}. From classical LWP results, see for instance \cite{linares} or \cite{MV}, the Benjamin equation is globally well-posed in $ H^1(\R) $ with uniqueness in $ L^\infty_{loc} H^1(\R)) $ and even in $ L^\infty_{loc} L^2(\R)) $. Moreover, for any $u_0\in H^1(\R) $, the associated solution $ u $ is bounded in $ H^1(\R) $ with a bound that depends on $ \|u_0\|_{H^1} $. So let $\{u_{0,n}\}_{n\ge 0}\subset H^1(\R) $ such that 
$ u_{0,n} \rightharpoonup u_0 $ in $ H^1(\R) $. From Banach-Steinhaus Theorem, $\{u_{0,n}\} $ is bounded in $ H^1(\R) $ and thus the sequence of emanating solutions $\{u_n\} $ is bounded in $C(\R; H^1(\R))$. In view of the equation \eqref{MainEq} it follows that the sequence $\{\partial_t u_n\} $ is bounded in $L^\infty(\R; H^{-2}(\R)) $. Aubin-Lions' compactness theorem then implies  that, up to a subsequence extraction, $\{u_n\} $ converges strongly in $ L^2_{loc}(\R;L^2_{loc}(\R)) $ to  some function $ v\in L^\infty(\R; H^1(\R))$. This allows us to pass to the limit on $u_n^2$  and ensures that $ v$ is a solution to \eqref{MainEq} belonging to the uniqueness class. Now let $ \phi\in C^\infty_c(\R) $. In view of the  above bound on $\{\partial_t u_n\} $, the family $\{t\mapsto (u_n,\phi)_{H^1}\} $ is bounded and uniformly equi-continuous on $ [-T,T] $ for any $T>0 $. It then follows from Ascoli's Theorem that $\{(u_n,\phi)_{H^1}\}$ converges to $(v,\phi)_{H^1} $ uniformly on $[-T,T] $. In particular, $v(0)=u_0 $ and thus  by the uniqueness result $ v\equiv u $ the unique solution to \eqref{MainEq} emanating from $ u_0 $ that belongs to $ L^\infty_{loc} H^1(\R) $. The uniqueness of the possible limit ensures  that the  above convergence holds actually for the sequence $\{u_n\} $ and not only for a subsequence.
Finally, since $ C^\infty_c(\R) $ is densely embedded in $ H^1(\R)$, this proves that for any $ \phi\in H^1(\R) $ and any $ T>0 $,  
$(u_n,\phi)_{H^1} \to (u,\phi)_{H^1} $ in $ C([-T,T] $ which is the desired convergence result. 
\subsection{Proof of Lemma \ref{LemdecayH}}
According to Proposition \ref{H1Monot} for $ t\ge t_0 $, it holds 
\begin{align*} 
\int_{\R}  \Bigl(\frac{1}{2} u_x^2 +\frac{\delta_0}{2} u^2- \frac{\gamma}{2} u \H u_x & - \frac{1}{6} \ u^3  \Bigr)(t) \tilde{\Psi}(t) \\ 
&\le \underbrace{\int_{\R} \Bigl(\frac{1}{2} u_x^2 +\frac{\delta_0}{2} u^2- \frac{\gamma}{2} u \H u_x \  - \frac{1}{6} \ u^3  \Bigr)(t_0) \tilde{\Psi}(t_0)}_{=H^{R,r}(u(t_0, \cdot+x(t_0))}  + C \, R^{-1/4} 
\end{align*}
where $ \tilde{\Psi} $ is defined in \eqref{defPsi2}. It thus suffices to prove that 
\begin{align}
\int_{\R}  \Bigl(\frac{1}{2} u_x^2  +&\frac{\delta_0}{2} u^2- \frac{\gamma}{2} u \H u_x  - \frac{1}{6} \ u^3  \Bigr)(t) \tilde{\Psi}(t)\nonumber \\
& \ge \int_{\R}  \Bigl(\frac{1}{2} u_x^2  +\frac{\delta_0}{2} u^2- \frac{\gamma}{2} u \H u_x  - \frac{1}{6} \ u^3  \Bigr)(t) \chi\Bigl(\frac{\cdot-(x(t)-2R)}{R^{3/4}}\Bigr)-C R^{-3/4} \; . \label{frG1}
 \end{align}
Let us start by proving that for $ R\ge 1$ and $t\ge t_0 $,  
\begin{equation} \label{frG2}
  \tilde{\Psi}(t,\cdot)= \chi \Bigl(\frac{\cdot -(x(t_0)-2R+\vartheta (t-t_0))}{(R+\frac{1}{8}(t-t_0))^{3/4}} \Bigr)  \ge \chi\Bigl(\frac{\cdot-(x(t)-2R)}{R^{3/4}}\Bigr) \quad \text{on} \; \mathbb{R} \; .
  \end{equation}
  We recall that $ \chi $ is  a non increasing function   with $\chi =0 $ on $\mathbb{R}_- $ and $ \chi=1 $ on$ [1,+\infty[ $. In particular, $ \chi(\frac{x-(x(t)-2R)}{R^{3/4}})=0 $ for $ x\le x(t) -2R $.  In addition, we have $ \tilde{\Psi}(t,x)=1 $  for $ x\ge x(t)-2R+R^{3/4} $ since, for $ R\ge 1 $, 
  \begin{align*}
  x-(x(t_0)-2R+\vartheta (t-t_0))&= x-(x(t)-2R+\vartheta (t-t_0)+(x(t_0)-x(t)) \\
   &\ge R^{3/4} +\frac{1}{8} (t-t_0)\ge  (R +\frac{1}{8} (t-t_0))^{3/4} \; ,
 \end{align*}
 where in the next to the last step we used the fact that  
 $ \theta (t_0-t)+(x(t_0)-x(t))\le -\frac{1}{8} (t-t_0)$ since $ \vartheta\le 5/8 $ and $ \dot{x}\ge 5/6 $. Therefore, it suffices to check that for $ x\in [x(t)-2R, x(t)-2R+R^{3/4}] $, 
 $$
 \frac{x-(x(t_0)-2R+\vartheta (t-t_0))}{(R +\frac{1}{8} (t-t_0))^{3/4}} \ge \frac{x-(x(t)-2R)}{R^{3/4}} \ .
 $$
 \begin{equation}\label{fr}
 \Leftrightarrow 
\Bigl(  x-(x(t_0)-2R+\vartheta (t-t_0))\Bigr) R^{3/4} \ge (x-(x(t)-2R))(R +\frac{1}{8} (t-t_0))^{3/4}
 \end{equation}
For this we notice that on one hand, proceeding as above, 
$$
\Bigl(  x-(x(t_0)-2R+\vartheta (t-t_0))\Bigr) R^{3/4}  \ge \Bigl( x-x(t)+2R +\frac{1}{8} (t-t_0) \Bigr) R^{3/4} 
$$
whereas on  the other hand, for $ x\le x(t)-2R+R^{3/4}$, 
 \begin{align*}
 (x-(x(t)-2R))(R +& \frac{1}{8} (t-t_0))^{3/4}\le  (x-(x(t)-2R))\Bigl(R^{3/4}  +\frac{1}{8} (t-t_0) \Bigr)\\
 & \le  (x-x(t)+2R)R^{3/4} +  \frac{R^{3/4}}{8} (t-t_0) \:.
 \end{align*}
Combining these last two estimates, \eqref{fr} follows which thus completes the proof of \eqref{frG2}. Now, we notice that \eqref{frG1} will follow directly from \eqref{frG2} if we could check that 
$$
\frac{1}{2} u_x^2 +\frac{\delta_0}{2} u^2- \frac{\gamma}{2} u \H u_x - \frac{1}{6} \ u^3  \ge 0 \quad \text{on} \quad \R \; .
$$
However, this inequality is not clear because of the presence of the non-local operator $ \H $. To overcome this difficulty, we first use the same trick as in \eqref{estgg} to rewrite  $\int_{\R} \tilde{\Psi} u_x^2 $ in the following way :
\begin{align}
\int_{\R}  \tilde{\Psi} u_x^2  & = \int_{\R} (\sqrt{ \tilde{\Psi}} u_x)(\sqrt{ \tilde{\Psi}}  u_x)= \int_{\R} \H(\sqrt{ \tilde{\Psi}} u_x ) \H(\sqrt{ \tilde{\Psi}} u_x)  \nonumber \\
& =\int_{\R}  \tilde{\Psi} (\H u_x)^2+2 \int_{\R} \sqrt{ \tilde{\Psi}} \H u_x [\H, \sqrt{ \tilde{\Psi}}]u_x +  \int_{\R} ( [\H, \sqrt{ \tilde{\Psi}}]u_x)^2 \; .
\label{coma}
\end{align} 
Noticing that  \eqref{comut} together with \eqref{deriv2} ensure that   $\|[\sqrt{\tilde{\Psi}}, \H] u_x\|_{L^2}\lesssim R^{-3/4}\|u_x\|_{L^2}  $, we thus get \begin{align}
&\int_{\R}  \Bigl(\frac{1}{2} u_x^2  +\frac{\delta_0}{2} u^2- \frac{\gamma}{2} u \H u_x  - \frac{1}{6} \ u^3  \Bigr)(t)  \tilde{\Psi}(t)\nonumber \\
&=  \int_{\R}  \Bigl(\frac{1}{2} (\H u_x)^2  +\frac{\delta_0}{2} u^2- \frac{\gamma}{2} u \H u_x  - \frac{1}{6} \ u^3  \Bigr)(t)\tilde{\Psi}(t)+ \|u\|_{H^1}^2 O (R^{-3/4}) \;\label{coco} .\end{align}
Now since $\|u\|_{L^\infty}\le  \|u\|_{H^1} \lesssim 1 $, by Young inequality it is direct to see  that for $ \delta_0 \ge 4 $ large enough, 
$$
\frac{1}{2} (\H u_x)^2  +\frac{\delta_0}{2} u^2- \frac{\gamma}{2} u \H u_x  - \frac{1}{6} \ u^3\ge 0 \quad \text{on} \quad \R
$$
and thus according to  \eqref{frG2}, \eqref{coma} and \eqref{coco},  we have the following chain of inequalities 
\begin{align*}
&\int_{\R}  \Bigl(\frac{1}{2} u_x^2  +\frac{\delta_0}{2} u^2- \frac{\gamma}{2} u \H u_x  - \frac{1}{6} \ u^3  \Bigr)(t)  \tilde{\Psi}(t)\\
&=\int_{\R}  \Bigl(\frac{1}{2} (\H u_x)^2  +\frac{\delta_0}{2} u^2- \frac{\gamma}{2} u \H u_x  - \frac{1}{6} \ u^3  \Bigr)(t)  \tilde{\Psi}(t) + O (R^{-3/4})\\
&\ge   \int_{\R}  \Bigl(\frac{1}{2} (\H u_x)^2  +\frac{\delta_0}{2} u^2- \frac{\gamma}{2} u \H u_x  - \frac{1}{6} \ u^3  \Bigr)(t)
 \chi\Bigl(\frac{\cdot-(x(t)-2R)}{R^{3/4}}\Bigr) +  O (R^{-3/4}) \\
 &\ge   \int_{\R}  \Bigl(\frac{1}{2} u_x^2  +\frac{\delta_0}{2} u^2- \frac{\gamma}{2} u \H u_x  - \frac{1}{6} \ u^3  \Bigr)(t)
 \chi\Bigl(\frac{\cdot-(x(t)-2R)}{R^{3/4}}\Bigr) +  O (R^{-3/4})  \end{align*}
 that proves \eqref{frG1} and completes the proof of the lemma.
 %\subsection{Proof of the decay estimate \eqref{tr3}}
% As explained in the proof of Proposition \ref{der}, we just have to check that the constant $ M_0 $ appearing in \eqref{tr3} holds for all $(\gamma,c)\in ]-\gamma_0,\gamma_0[\times ]1-\delta_0,1+\delta_0[ $. According to (\cite{ADM1}, proof of Proposition 4.1), it reduces to check the decay of 
%% $$
%\displaystyle G(x)  = \frac{1}{\sqrt{2 \pi}} \int_{\mathbb{R}^+} \frac{\cos (x \xi)}{c+  \xi^2 + \gamma \xi} \ d\xi \
%$$
%that can also be expressed by Cauchy's Integral Theorem   as 
%\begin{align*}
 %G(x)   =&\frac{1}{\sqrt{2\pi}}\lim_{R\to +\infty} \Bigl(  +e^{-\theta|x|/2}  \int_0^R  \frac{e^{i |x| t }}{c+(t+i\frac{\theta}{2})^2 + \gamma (t + i\frac{\theta}{2})} \ dt
 %+i \int_0^{\frac{\theta}{2}} \frac{(c-t^2) - i \gamma t}{(c-t^2)^2 + \gamma^2 t^2} e^{-|x| t}\, dt \Bigr) \\
% &=I_1+I_2 \; .
%\end{align*} 
%where,  for $ \gamma_0\le 1$, we can take  $ \theta=\sqrt{4-\gamma^2}/2 $. In particular, $\sqrt{3}/2\le \theta \le 1$ and thus for $ c>3/4 $ on can check that
%$$
%|c+(t+i\frac{\theta}{2})^2 + \gamma (t + i\frac{\theta}{2})|\ge c-\theta^2/4 +t^2+\gamma t =(t+\frac{\gamma}{2})^2+(c-\theta^2/4-\gamma^2/4) \ge 
%(t+\frac{\gamma}{2})^2+1/4 \; .
%$$
%This ensures that there exists $ C>0 $ such that  $ |I_1| \le C \, e^{-\sqrt{3} |x|/4} $.

\end{document}